\documentclass[reqno,10pt]{amsart}
\usepackage{amsmath,amsthm,amsfonts,amssymb,amscd}

\theoremstyle{plain}
\newtheorem{maintheorem}{Theorem}

\newtheorem{theorem}{Theorem }[section]
\newtheorem{proposition}[theorem]{Proposition}
\newtheorem{lemma}[theorem]{Lemma}
\newtheorem{corollary}[theorem]{Corollary}

\theoremstyle{definition}%\theoremstyle{remark}
\newtheorem{remark}[theorem]{Remark}

\newtheorem{problem}[theorem]{Problem}

\newcommand{\C}{\ensuremath{\mathbb{C}}}
\newcommand{\R}{\ensuremath{\mathbb{R}}}

\newcommand{\Z}{\ensuremath{\mathbb{Z}}}
\newcommand{\N}{\ensuremath{\mathbb{N}}}

\newcommand{\sps}{p}
\newcommand{\sqs}{q}
\newcommand{\Proj}{\mathbb{P}}
\newcommand{\bx}{{\mathbf{x}}}
\newcommand{\by}{{\mathbf{y}}}

\newcommand{\sx}{x}
\newcommand{\sy}{y}

\newcommand{\sA}{A}
\newcommand{\sB}{B}
\newcommand{\sC}{C}
\newcommand{\sD}{D}

\newcommand{\sK}{K}
\newcommand{\mC}{c}

\newcommand{\const}{\operatorname{const}}
\newcommand{\vep}{\varepsilon}
\newcommand{\supp}{\operatorname{supp}}

\newcommand{\cB}{{\mathcal B}}

\newcommand{\cD}{{\mathcal D}}

\newcommand{\cF}{{\mathcal F}}
\newcommand{\cG}{{\mathcal G}}

\newcommand{\cK}{{\mathcal K}}
\newcommand{\cL}{{\mathcal L}}
\newcommand{\cM}{{\mathcal M}}

\newcommand{\cP}{{\mathcal P}}
\newcommand{\cQ}{{\mathcal Q}}

\newcommand{\cS}{{\mathcal S}}

\newcommand{\cV}{{\mathcal V}}

\newcommand{\cX}{{\mathcal X}}

\newcommand{\cY}{{\mathcal Y}}
\newcommand{\cZ}{{\mathcal Z}}

\newcommand{\Stat}{\operatorname{Stat}}
\newcommand{\diam}{\operatorname{diam}}
\newcommand{\id}{\operatorname{id}}
\newcommand{\CC}{\mathbb{C}}
\newcommand{\RR}{\mathbb{R}}

\newcommand{\GL}{\operatorname{GL}}
\newcommand{\SL}{\operatorname{SL}}
\newcommand{\prob}{\operatorname{probability}}
\newcommand{\quand}{\quad\text{and}\quad}

\newcommand{\nsubset}{{\not\subset}}

\begin{document}
\title[Continuity of Lyapunov exponents]
{Continuity of Lyapunov exponents \\ for random 2D matrices}
\author{Carlos Bocker-Neto and Marcelo Viana}
\thanks{C.B.-N. was supported by a CNPq and FAPERJ doctoral scholarship.
M.V. is partially supported by CNPq, FAPERJ, and PRONEX-Dynamical Systems.}
\address{IMPA, Est. D. Castorina 110 \\ 22460-320 Rio de Janeiro, RJ, Brazil}
\email{bocker@impa.br, viana@impa.br}
\date{}
\begin{abstract}
The Lyapunov exponents of locally constant $\GL(2,\CC)$-cocycles over
Bernoulli shifts depend continuously on the cocycle and on the invariant
probability. The Oseledets decomposition also depends continuously on the
cocycle, in measure.
\end{abstract}

\maketitle

\tableofcontents \newpage

\section{Introduction}

Let $A_1$, \dots, $A_m$ be invertible $2$-by-$2$ matrices
and $p_1$, \dots, $p_m$ be (strictly) positive numbers with $p_1+\cdots+p_m=1$.
Consider
$$
L^n = L_{n-1}\cdots L_1 L_0, \quad n\ge 1,
$$
where the $L_j$ are independent random variables with identical probability
distributions, given by
$$
\prob(\{L_j=A_i\}) = p_i \quad\text{for all $j\ge 0$ and $i=1, \dots, m$.}
$$
It is a classical fact, going back to Furstenberg, Kesten~\cite{FK60}, that there
exist numbers $\lambda_+$ and $\lambda_-$ such that
\begin{equation}\label{eq.FK}
\lim_{n\to\infty}\frac 1n \log \|L^n\| = \lambda_+
\quand
\lim_{n\to\infty}\frac 1n \log \|(L^n)^{-1}\|^{-1} = \lambda_-
\end{equation}
almost surely. The results in this paper imply that these \emph{extremal Lyapunov
exponents} always vary continuously with the choice of the matrices and the
probability weights:

\begin{maintheorem}\label{t.finitecase}
The extremal Lyapunov exponents $\lambda_+$ and $\lambda_-$ depend
continuously on $(A_1, $ $\dots,$ $A_m,p_1,$ $\dots, p_m)$ at all points.
\end{maintheorem}

This conclusion holds in much more generality. Indeed, we may take the probability
distribution of the random variables $L_j$ to be any probability measure $\nu$ on
$\GL(2,\CC)$ with compact support. Let $\lambda_+(\nu)$ and $\lambda_-(\nu)$,
respectively, denote the values of the (almost certain) limits in \eqref{eq.FK}.
Then we have:

\begin{maintheorem}\label{t.generalcase}
For every $\vep>0$ there exists $\delta>0$ and a weak$^*$ neighborhood $V$ of $\nu$ in
the space of probability measures on $\GL(2,\CC)$ such that
$|\lambda_\pm(\nu)-\lambda_\pm(\nu')| < \vep$ for every probability measure $\nu'\in V$
whose support is contained in the $\delta$-neighborhood of the support of $\nu$.
\end{maintheorem}

The situation in Theorem~\ref{t.finitecase} corresponds to the special case when the
measures have finite supports:
$$
\nu = p_1 \delta_{A_1} + \cdots + p_m \delta_{A_m}
\quand
\nu' = p_1' \delta_{A_1'} + \cdots + p_m' \delta_{A_m'}.
$$
Clearly, the support of $\nu'$ is Hausdorff close to the support of $\nu$ if $A_i'$ is close
to $A_i$, $p_i$ for all $i$. In this regard, recall that we assume that all $p_i>0$:
the conclusion of Theorem~\ref{t.finitecase} may fail if this condition is removed,
as we will recall in Remark~\ref{r.vanishingpi}.

\section{Continuity of Lyapunov exponents}

In this section we put the previous results in a broader context and give a convenient
translation of Theorem~\ref{t.generalcase} to the theory of linear cocycles.

\subsection{Linear cocycles}

Let $\pi:\cV\to M$ be a  finite-dimensional (real or complex) vector bundle and
$F:\cV\to\cV$ be a \emph{linear cocycle} over some measurable transformation $f:M\to M$.
By this we mean that $\pi \circ F = f \circ \pi$ and the actions $F_x:\cV_x\to\cV_{f(x)}$
on the fibers are linear isomorphisms.
Take $\cV$ to be endowed with some measurable Riemannian metric, that is, an Hermitian
product on each fiber depending measurably on the base point.
Let $\mu$ be an $f$-invariant probability measure on $M$ such that
$$
\log\|(F_x)^{\pm 1}\| \in L^1(\mu).
$$
It follows from the sub-additive ergodic theorem (Kingman~\cite{Ki68})
that the numbers
$$
\lambda_+(F,x)=\lim_{n\to\infty}\frac 1n \log \|F_x^n\|
\quand
\lambda_-(F,x)=\lim_{n\to\infty}\frac 1n \log \|(F_x^n)^{-1}\|^{-1}
$$
are well-defined $\mu$-almost everywhere.

The theorem of Oseledets~\cite{Ose68} provides a more detailed statement.
Namely, at $\mu$-almost every point $x \in M$, there exist numbers
$$
\hat\lambda_1(F,x) > \cdots > \hat\lambda_{k(x)}(F,x)
$$
and a filtration
\begin{equation}\label{eq.filtration}
\cV_x = V_x^1 > V_x^2 > \cdots > V_x^{k(x)}> \{0\} = V_x^{k(x)+1}
\end{equation}
such that $F_x(V_x^j)=V_{f(x)}^j$ and
$$
\lim_{n\to\infty} \frac 1n \log\|F_x^n(v)\| = \hat\lambda_j(F,x)
\quad\text{for all $v \in V_x^j \setminus V_x^{j+1}$.}
$$
When $f$ is invertible one can say more: there exists a splitting
\begin{equation}\label{eq.splitting}
\cV_x = E_x^1 \oplus E_x^2 \oplus \cdots \oplus E_x^{k(x)}
\end{equation}
such that $F_x(E_x^j)=E_{f(x)}^j$ and
$$
\lim_{n\to\pm \infty} \frac 1n \log\|F_x^n(v)\| = \hat\lambda_j(F,x)
\quad\text{for all $v \in E_x^j \setminus\{0\}$.}
$$
The number $k(x)\ge 1$ and the \emph{Lyapunov exponents} $\hat\lambda_j(F,\cdot)$
are measurable functions of the point $x$, with
$$
\hat\lambda_1(F,x)=\lambda_+(F,x) \quand \hat\lambda_{k(x)}(F,x) = \lambda_-(F,x),
$$
and they are constant on the orbits of $f$. In particular, they are
constant $\mu$-almost everywhere if $\mu$ is ergodic.

\subsection{Continuity problem}

Next, let $\lambda_1(F,x) \ge \cdots \ge \lambda_d(F,x)$ be the list of all Lyapunov
exponents, where each is counted according to its
multiplicity $m_j(x)=\dim V_x^j - \dim V_x^{j+1}$ ($= \dim E_x^j$ in the
invertible case). Of course, $d = $ dimension of $\cV$.
The \emph{average Lyapunov exponents} of $F$ are defined by
$$
\lambda_i(F,\mu)=\int \lambda_i(F,\cdot)\,d\mu,
\quad \text{for $i=1, \dots, d$.}
$$
The results in this paper are motivated by the following basic question:

\begin{problem}%[Continuity Problem]
What are the continuity points of
$$
(F,\mu) \mapsto (\lambda_1(F,\mu), \dots, \lambda_d(F,\mu))~?
$$
\end{problem}

It is well known that the sum of the $k$ largest Lyapunov exponents
\begin{equation}\label{eq.inf1}
(F,\mu) \mapsto \lambda_1(F,\mu)+\cdots + \lambda_k(F,\mu)
\end{equation}
(any $1 \le k < d$) is upper semi-continuous, relative to the $L^\infty$-norm in the space of cocycles
and the pointwise topology in the space of probabilities (the smallest topology that makes
$\mu\mapsto\int\psi\,d\mu$ continuous for every bounded measurable function $\psi$).
Indeed, this is an easy consequence of the identity
\begin{equation}\label{eq.inf2}
\lambda_1(F,\mu)+\cdots + \lambda_k(F,\mu) = \inf_{n\ge1} \frac 1n \int \log\|\Lambda^k(F^n_x)\|\,d\mu(x)
\end{equation}
where $\Lambda^k$ denotes the $k$th exterior power. Similarly, the sum of the $k$ smallest
Lyapunov exponents is always lower semi-continuous.
However, Lyapunov exponents are, usually, \emph{discontinuous} functions of the data.
A number of results, both positive and negative, will be recalled in a while.

\subsection{Continuity theorem}

Let $\cX$ be a polish space, that is, a separable completely metrizable topological space.
Let $\sps$ be a probability measure on $\cX$
%consequently $M=\cX^{\Z}$ is also a polish space [see \cite{Gr09}, for instance].
%Let $\sps$ a Borel probability measure on $\cX$
%Let $\cX$ be a measurable space, $\sps$ a probability on $\cX$ such that $(\cX,\sps)$ is a Lebesgue space
and $\sA:\cX\to\GL(2,\C)$ be a measurable function such that
\begin{equation}\label{eq.bounded}
\log\|A^{\pm 1}\| \quad\text{are bounded.}
\end{equation}
Let $f:M\to M$ be the shift map on $M$ and $\mu=\sps^\Z$. Consider the linear cocycle
$$
F:M \times \C^2 \to M \times \C^2, \quad
F(\bx,v) = (f(\bx), A_{\sx_0}(v)),
$$
where $x_0\in\cX$ denotes the zeroth coordinate of $\bx\in M$.
%Define $\|A-B\|=\sup_{x\in X} \|A(x)-B(x)\|$ for any measurable $A, B:X\to\GL(2,\CC)$.
In the spaces of cocycles and probability measures on $\cX$ we consider the distances defined by,
respectively,
\begin{equation}\label{eq.distances}
d(\sA,\sB)  = \sup_{x\in \cX} \|A_\sx-B_\sx\| \qquad
d(\sps,\sqs)   % = \|\sps-\sqs\|
= \sup_{|\phi|\le 1} |\int \phi\,d(\sps - \sqs)|
\end{equation}
where the second $\sup$ is over all measurable functions $\phi:\cX\to\R$ with
$\sup|\phi|\le 1$.
In the space of pairs $(\sA,\sps)$ we consider the topology
determined by the bases of neighborhoods
\begin{equation}\label{eq.topology2}
V(\sA,\sps,\gamma,\cZ)=\{(\sB,\sqs): d(\sA,\sB)<\gamma,\ \sqs(\cZ)=1, \ d(\sps,\sqs)<\gamma\}
\end{equation}
where $\gamma>0$ and  $\cZ$ is any measurable subset of $\cX$ with $\sps(\cZ)=1$.

\begin{maintheorem}\label{t.Bernoulli}
The extremal Lyapunov exponents $\lambda_\pm(\sA,\sps)=\lambda_\pm(F,\mu)$ depend continuously on
$(\sA,\sps)$ at all points.
\end{maintheorem}

We prove Theorem~\ref{t.Bernoulli} in Sections~\ref{s.proof1} and~\ref{s.keyproof}, and
we deduce Theorem~\ref{t.generalcase} from it in Section~\ref{s.consequences}.
Theorem~\ref{t.Bernoulli} can also be deduced from Theorem~\ref{t.generalcase}:
if $d(\sA,\sB)$ and $d(\sps,\sqs)$ are small then $\nu'=\sB_*\sqs$ is close to $\nu=\sA_*\sps$
in the weak$^*$ topology, and the support of $\nu'$ is contained in a small neighborhood of
the support of $\nu$; moreover,
$\lambda_\pm(\sA,\sps)=\lambda_\pm(\nu)$ and $\lambda_\pm(\sB,\sqs)=\lambda_\pm(\nu')$.
In this way one even gets a more general version of Theorem~\ref{t.Bernoulli},
where $\cX$ can be any measurable space.

Our arguments also show that the Oseledets decomposition depends continuously on the cocycle
in measure. Given $B:\cX\to\GL(2,\CC)$, let $E^s_{B,\bx}$ and $E^u_{B,\bx}$ be the Oseledets
subspaces of the corresponding cocycle at a point $\bx\in M$ (when they exist).

\begin{maintheorem}\label{t.convergenciaemmedida}
Suppose $\lambda_-(\sA,\sps)<\lambda_+(\sA,\sps)$. For any sequence $\sA^k:\cX\to\GL(2,\CC)$
such that $d(A^k,A)\to 0$, and for any $\vep>0$, we have
$$
\mu\big(\{x\in M: \angle(E_{\sA,x}^u,E_{\sA^k,x}^{u})<\vep\text{ and }\angle(E_{\sA,x}^s,E_{\sA^k,x}^{s})<\vep \}\big)
\to 1.
$$
%In particular, if $\sps_k=\sps$ for all $k$ then the functions $(\bx\mapsto E^{*}_{\sA^k,x})$
%converge in $\sps^{\Z}$-measure to $(\bx\mapsto E^*_{\sA,\bx})$, for both $*\in\{s,u\}$.
\end{maintheorem}

A few words are in order on our choice of the topology \eqref{eq.topology2}. As we are going to see,
the proof of Theorem~\ref{t.Bernoulli} splits into two cases, depending on whether the cocycle is
almost irreducible (Section~\ref{ss.proofnaodiagonal}) or diagonal (Section~\ref{ss.proofdiagonal}).
In the irreducible case, continuity of the Lyapunov exponents was known before (\cite{FK83,Hen84},
see also~\cite{AV3}) and only requires the weak$^*$ topology. In a nutshell, this is because in the
irreducible case
\begin{equation}\label{eq.furstenbergformula}
\lambda_+(\sA,\sps)= \int \log\frac{\|\sA(\bx)(v)\|}{\|v\|}\,d\mu(\bx)\,d\eta(v)
\end{equation}
for \emph{every} stationary measure $\eta$ (Furstenberg's formula); then one only has to note that
the set of stationary measures varies semi-continuously with the data.
The main point in the proof Theorem~\ref{t.Bernoulli} is to handle the diagonal case,
where \eqref{eq.furstenbergformula} breaks down. That is where we need the full strength of
\eqref{eq.topology2}.

Restricted to the space of pairs $(A,\sps)$ where $A$ is continuous (and bounded), it suffices to
consider the neater bases of neighborhoods
\begin{equation}\label{eq.topology1}
V(\sA,\sps,\vep)=\{(\sB,\sqs): d(\sA,\sB)<\vep,\ \supp\sqs\subset \supp\sps, \ d(\sps,\sqs)<\vep\}.
\end{equation}
However, this will not be used in the present paper.

\subsection{Previous results}

The problem of dependence of Lyapunov exponents on the linear cocycle or the base dynamics
has been addressed by several authors. In a pioneer work, Ruelle~\cite{Ru79} proved real-analytic
dependence of the largest exponent on the cocycle, for linear cocycles admitting an invariant
convex cone field. Short afterwards, Furstenberg, Kifer~\cite{FK83,Ki82b} and Hennion~\cite{Hen84}
proved continuity of the largest exponent of i.i.d. random matrices, under a condition of almost
irreducibility. Some reducible cases were treated by Kifer and Slud~\cite{Ki82b,Ki82a},
who also observed  that discontinuities may occur when the probability vector degenerates
(\cite{Ki82b}, cf. Remark~\ref{r.vanishingpi} below).

For i.i.d. random matrices satisfying strong irreducibility and the contraction property,
Le Page~\cite{LP82,LP89} proved local H\"older continuous, and even smooth, dependence of the
largest exponent on the cocycle; the assumptions ensure that the largest exponent is simple
(multiplicity $1$), by work of Guivarc'h, Raugi~\cite{GR86} and Gol'dsheid, Margulis~\cite{GM89}.
For i.i.d. random matrices over Bernoulli and Markov shifts, Peres~\cite{Pe91} showed that simple
exponents are locally real-analytic functions of the transition data.

A construction of Halperin  quoted by Simon, Taylor~\cite{ST85} shows that for every $\alpha>0$
one can find \emph{random Schr\"odinger cocycles}
$$
\left(\begin{array}{cc} E - V_n & -1 \\ 1 & 0 \end{array}\right)
$$
(the $V_n$ are i.i.d. random variables) near which the exponents fail to be $\alpha$-H\"older
continuous. Thus, the previously mentioned results of Le Page can not be improved.
Johnson~\cite{Jo84} found examples of discontinuous dependence of the exponent on the
energy $E$, for Schr\"odinger cocycles over quasi-periodic flows. Recently, Bourgain,
Jitomirskaya~\cite{Bou05,BJ02b} proved continuous dependence of the exponents on the
energy $E$, for one-dimensional \emph{quasi-periodic} Schr\"odinger cocycles:
$V_n=V(f^n(\theta))$ where $V:S^1\to\RR$ is real-analytic and $f$ is an irrational circle
rotation.

Going back to linear cocycles, the answer to the continuity problem is bound to depend on the
class of cocycles under consideration, including its topology. Knill~\cite{Kni91,Kni92} considered
$L^\infty$ cocycles with values in $\SL(2,\RR)$ and proved that, as long as the base dynamics
is aperiodic, discontinuities always exist: the set of cocycles with non-zero exponents is
never open. This was refined to the continuous case by Bochi~\cite{Boc-un,Boc02}:
an $\SL(2,\RR)$-cocycle is a continuity point in the $C^0$ topology if and only if it is
uniformly hyperbolic or else the exponents vanish. This statement was inspired by Ma\~n\'e's
surprising announcement in~\cite{Man83}. Indeed, and most strikingly, the theorem of
Ma\~n\'e-Bochi~\cite{Boc02,Man83} remains true restricted to the subset of $C^0$ derivative cocycles,
that is, of the form $F=Df$ for some $C^1$ area preserving  diffeomorphism $f$.
Moreover, this has been extended to cocycles and diffeomorphisms in arbitrary dimension,
by Bochi, Viana~\cite{Boc09,BcV05}.
Let us also note that  linear cocycles whose exponents are all equal form an $L^p$-residual
subset, for any $p\in[1,\infty)$, by Arnold, Cong~\cite{AC97}, Arbieto, Bochi~\cite{ArB03}.
Consequently, they are precisely the continuity points for the Lyapunov exponents relative to
the $L^p$ topology.

These results show that discontinuity of Lyapunov exponents is quite common among cocycles with
low regularity. Locally constant cocycles, as we deal with here, sit at the opposite end of the
regularity spectrum, and the results in the present paper show that in this context continuity
does hold at every point. For cocycles with intermediate regularities the continuity problem is
very much open. However, our construction in Section~\ref{ss.discontinuity} shows that for any
$r\in(0,\infty)$ there exist locally constant cocycles over Bernoulli shifts that are points of
discontinuity for the Lyapunov exponents in the space of all $r$-H\"older cocycles. We will return
to this topic in the final section.

Recently, Avila, Viana~\cite{AV3} studied the continuity of the Lyapunov exponents in the very broad
context of \emph{smooth} cocycles. The continuity criterium in~\cite[Section~5]{AV3} was the starting
point for the proof of our Theorem~\ref{t.Bernoulli}.

\medskip

This paper is organized as follows.
In Section~\ref{s.proof1} we reduce Theorem~\ref{t.Bernoulli} to a key result on stationary
measures of nearby cocycles. The latter is proved in Sections~\ref{s.keyproof} and \ref{s.estimates}.
In Section~\ref{s.consequences} we deduce Theorems~\ref{t.generalcase}
and~\ref{t.convergenciaemmedida}.
Finally, in Section~\ref{s.final} we describe an example of discontinuity of Lyapunov
exponents for H\"older cocycles, and we close with a short list of open problems and
conjectures.

\subsubsection*{Acknowledgements} We are grateful to Artur Avila, Jairo Bochi, and Jiagang Yang
for several useful conversations. Lemma~\ref{l.pullback} is due to Artur Avila.

\section{Proof of Theorem~\ref{t.Bernoulli}}\label{s.proof1}

We start with a simple observation. Let $\cP(\cX)$ be the space of probability measures
on $\cX$ and let $\cG(\cX)$ and $\cS(\cX)$ denote the spaces of bounded measurable
functions from $\cX$ to $\GL(2,\CC)$ and $\SL(2,\CC)$, respectively.
Given any $\sA\in \cG(X)$ let $\sB\in\cS(\cX)$
and $\mC:\cX\to\CC$ be such that $A_x=c_x B_x$ for every $x\in\cX$.
Although $c_x=(\det A_x)^{1/2}$ and $B_x$ are determined up to sign only, choices can be
made consistently in a neighborhood, so that $\sB$ and $\mC$ depend continuously on $\sA$.
It is also easy to see that the Lyapunov exponents are related by
$$
\lambda_\pm(\sA,\sps)
 = \lambda_\pm(\sB,\sps) + \int \log |c_x|\,d\sps(x)
% = \lambda_\pm(\sB,\sps) + \frac 12 \sum_{i\in\cI} p_i \log |\det A_i|.
$$
Thus, since the last term depends continuously on $(\sA,\sps)$ relative to the topology
defined by \eqref{eq.topology2}, continuity of the Lyapunov exponents on
$\cS(\cX)\times \cP(\cX)$ yields continuity on the whole $\cG(X)\times \cP(\cX)$.
So, we may suppose from the start that $\sA\in\cS(\cX)$.
Observe also that in this case one has
$$
\lambda_+(\sA,\sps)+\lambda_-(\sA,\sps)=0.
$$

From here on the proof has two main steps. First, we reduce the problem to the case when
the matrices are simultaneously diagonalizable:

\begin{proposition}\label{p.naodiagonal}
If $(\sA,\sps)\in\cS(X)\times\cP(\cX)$ is a point of discontinuity for $\lambda_+$ then
there is $P\in\SL(2,\CC)$ and $\theta:\cX\to\CC\setminus\{0\}$ such that
$$
P A_x P^{-1} =\left(\begin{array}{cc} \theta_x & 0 \\ 0 & \theta_x^{-1} \end{array}\right)
$$
for all $x\in\cZ$, where $\cZ\subset\cX$ is a full $\sps$-measure set.
In particular, $A_xA_y=A_yA_x$ for all $x, y \in \cZ$.
\end{proposition}

Then we rule out the diagonal case as well:

\begin{proposition}\label{p.diagonal}
Let $(\sA,\sps)\in\cS(X)\times\cP(\cX)$ be such that $\sA$ is as in the conclusion
of Proposition~\ref{p.naodiagonal}. Then $(\sA,\sps)$ is a point of continuity for $\lambda_+$.
\end{proposition}

The proofs of these two propositions are given in the next couple of sections.
In view of the previous observations, they contain the proof of Theorem~\ref{t.Bernoulli}.

\subsection{Reducing to the diagonal case}\label{ss.proofnaodiagonal}

The proof of Proposition~\ref{p.naodiagonal}
is a simplified version of ideas of Avila, Viana~\cite{AV3}, partly inspired by
Bonatti, Gomez-Mont, Viana~\cite{BGV03}. For the sake of completeness,
and also because our setting is not strictly contained in~\cite{AV3},
we give the full arguments. The definitions and preliminary results apply
to functions $A$ with values in $\GL(d,\CC)$, for any $d\ge 2$.

The \emph{local stable set} $W^s_{loc}(\bx)$ of $\bx\in M$ is the set of all
$\by=(\sy_n)_{n\in\Z}$ such that $\sx_n=\sy_n$ for all $n\ge 0$.
The \emph{local unstable set} $W^u_{loc}(\bx)$ is defined similarly,
considering $n < 0$ instead. The \emph{projective cocycle} associated to
$\sA:\cX\to\GL(d,\CC)$ is defined by
$$
F_\sA:M\times\Proj(\CC^d)\to M\times\Proj(\CC^d),
\quad (\bx,[v]) \mapsto (f(\bx),[A(\bx)v])
$$
where $A(\bx)= A_{\sx_0}$ for every $\bx\in M$.

\subsubsection{Invariant $u$-states}

Let $\cM(\sps)$ denote the set of probability measures in $M\times\Proj(\CC^d)$
that project down to $\mu$.
A \emph{disintegration} of $m\in\cM(\sps)$ is a measurable function
assigning to each point $\bx\in M$ a probability $m_\bx$ with
$m_\bx\big(\{\bx\}\times\Proj(\CC^d)\big)=1$ and such that
$$
m(E) = \int m_\bx(E)\,d\mu(\bx),
\quad\text{for every measurable $E\subset M\times\Proj(\CC^d)$}.
$$
A disintegration always exists in this setting; moreover, it is essentially unique.
See Rokhlin~\cite{Ro52} and \cite[Appendix~C.6]{Beyond}.

A probability $m\in\cM(\sps)$ is a \emph{$u$-state} if
some  disintegration $\bx\mapsto m_\bx$ is constant on every local unstable set,
restricted to a full $\mu$-measure subset of $M$.
Then the same is true for every disintegration, by essential uniqueness;
moreover, one can choose the disintegration so that it is constant on
local unstable sets on the whole $M$.
If $m$ is an invariant probability then we say that $m$ is an \emph{invariant $u$-state}.
The definition of \emph{invariant $s$-states} is analogous, considering local
stable sets instead, and the same observations apply.

An $su$-\emph{state} is a probability which is both a $u$-state and an $s$-state.

\begin{lemma}\label{l.suestados}
A probability $m\in\cM(\sps)$ is an invariant
$su$-state if and only if $m=\mu\times\eta$ for some probability measure $\eta$
on $\Proj(\CC^d)$ invariant under the action of $A_x$ for $\sps$-almost every $x\in\cX$.
\end{lemma}

\begin{proof}The ``if'' part is not used in this paper, so we leave the
proof to the reader.
To prove the "only if" part notice that, by assumption,
$m$ admits disintegrations $\bx\mapsto m_\bx^u$, constant on local
unstable sets, and $\bx\mapsto m_\bx^s$, constant on local stable sets.
By essential uniqueness, there exists a full $\mu$-measure set
$X\subset M$ such that $m_\bx^u=m_\bx^s$ for all $\bx\in X$.
The assumption on $\mu$ implies that
$\mu=\mu^u\times\mu^s$ where $\mu^u$ is a probability on the set
positive one-sided sequences $(\sx_n)_{n\ge 0}$ and $\mu^s$ is a
probability on the set negative one-sided sequences $(\sx_n)_{n<0}$.
Fix $\bar\bx\in M$ such that $W_{loc}^u(\bar\bx)$ intersects $X$ on
a full $\mu^u$-measure set. Then let $\eta = m^u_{\bar\bx}$.
The local stable sets through the points of $X\cap W^u_{loc}(\bx)$
fill-in a full $\mu$-measure subset of $M$. Thus, $\eta=m_\bx^s$
at $\mu$-almost every point and so the constant family
$\bx \mapsto m_\bx=\eta$ is a disintegration of $m$. This means
that $m=\mu\times\eta$. Finally, the fact that $\mu$ and $m$ are
invariant gives $A(\bx)_*m_\bx=m_{f(\bx)}$ at $\mu$-almost
every point and that implies $(A_x)_*\eta=\eta$ for $\sps$-almost every
$x \in \cX$, as claimed.
\end{proof}

\begin{lemma}\label{l.ledrappier}
If $\lambda_\pm(\sA,\sps)=0$ then every $F_\sA$-invariant measure $m$
in $\cM(\sps)$ is an $su$-state.
\end{lemma}

\begin{proof}
This is a direct consequence of Ledrappier~\cite[Theorem~1]{Le86}.
Indeed, let $\cB^{s}$ be the $\sigma$-algebra of measurable subsets of $M$
which are unions of entire local stable sets.
Clearly, $f$ and $F_\sA$ are $\cB^{s}$-measurable.
Hence, Ledrappier's theorem gives that the disintegration of any $F_\sA$-invariant
probability $m\in\cM(\sps)$ is $\cB^{s}$-measurable modulo zero $\mu$-measure sets.
This is the same as saying that $m$ is an $s$-state.
Analogously, one proves that $m$ is a $u$-state.
\end{proof}

Let us consider the function $\phi_\sA: M \times \Proj(\CC^d)\rightarrow\R$ defined by
\begin{equation}\label{eq.phiA}
\phi_A(\bx,[v])=\log\frac{\| A(\bx)v\|}{\| v\|}.
\end{equation}

\begin{lemma}\label{l.intervalo}
For every $\sA:\cX \to\GL(d,\CC)$ and every $F_\sA$-invariant probability
 measure $m\in\cM(\sps)$,
$$
\lambda_-(\sA,\sps) \le \int\phi_\sA\,dm \le \lambda_+(\sA,\sps).
$$
\end{lemma}

\begin{proof}
For every $(\bx,[v])\in M\times\Proj(\CC^{d})$ and $n\ge 1$,
$$
\sum_{j=0}^{n-1} \phi_{A}(F_{\sA}^j(\bx,[v])) \le \log\|A^n(\bx)\|\,.
$$
Integrating with respect to any probability $m\in\cM(\sps)$,
$$
\frac 1n \int \sum_{j=0}^{n-1} \phi_\sA\circ F_{\sA}^{j}\,dm
%\le \frac 1n \int \log\|A^n(x)\|\,dm(\bx,[v])
\le \frac 1n \int \log\|A^n(\bx)\|\,d\mu(\bx).
$$
The right hand side converges to $\lambda_{+}(\sA,\sps)$ and, assuming $m$ is invariant,
the left hand side coincides with $\int\phi_{\sA}\,dm$.
This gives the upper bound in the statement. The lower bound is analogous.
\end{proof}

Now let $\sA$ take values in $\SL(2,\CC)$.
We want to show that the upper bound in Lemma~\ref{l.intervalo} is attained at some $u$-state
and the lower bound is attained at some $s$-state.
When $\lambda_{\pm}(\sA,\sps)=0$ this is a trivial consequence of Lemma~\ref{l.ledrappier}.
So, it is no restriction to suppose that $\lambda_+(\sA,\sps)>0>\lambda_-(\sA,\sps)$.

Let $E^u_\bx \oplus E^s_\bx$ be the Oseledets splitting of $F_\sA$,
defined at $\mu$-almost every $\bx$.
Consider the probabilities $m^u$ and $m^s$ defined on $M\times \Proj(\CC^2)$ by
\begin{equation}\label{eq.mus}
m^*(B) =\mu\big(\{\bx: (\bx,E^*_\bx) \in B \}\big)
       =\int \delta_{(\bx,E^*_\bx)}(B)\,d\mu(\bx)
\end{equation}
for $*\in\{s, u\}$ and any measurable subset $B$.
It is clear that $m^u$ and $m^s$ are invariant under $F_\sA$ and project down to $\mu$.
Moreover, their disintegrations are given by
$$
\bx \mapsto \delta_{(\bx,E^*_\bx)} \quad \text{for } *\in\{s,u\}.
$$
Since $E^u_\bx$ depends only on $\{A_{\sx_n}: n < 0\}$ and $E^s_\bx$ depends only on
$\{A_{\sx_n}: n \ge 0\}$, we get that $m^u$ is a $u$-state and $m^s$ is an $s$-state.
%We refer to them as the \emph{canonical} $u$-state and $s$-state of $A$.

\begin{lemma}\label{l.combinacaolinear}
Every $F_\sA$-invariant probability measure $m\in\cM(\sps)$ is a convex
combination $m=\alpha m^u +\beta m^s$, for some $\alpha,\beta\geq 0$ with $\alpha+\beta=1$.
\end{lemma}

\begin{proof}
Given $\kappa > 0$, define $X_{\kappa}$ to be the set of all
$(\bx,[v])\in M\times\Proj(\CC^{2})$ such that the Oseledets splitting
$E^{u}_{\bx} \oplus E^{s}_{\bx}$ is defined at $\bx$ and $[v]$ splits
$v=v^{u}+v^{s}$ with $ \kappa^{-1}\|v^{s}\| \le \|v^{u}\| \le \kappa \|v^{s}\|$.
Since the two Lyapunov exponents are distinct, any point of $X_{\kappa }$
returns at most finitely many times to $X_{\kappa}$. So,
by Poincar\'e recurrence, $m(X_{\kappa})=0$ for every $\kappa$.
This means that $m$ gives full weight to $\{(\bx,E^{u}_{\bx}), (\bx,E^{s}_{\bx}): \bx\in M\}$
and so it is a convex combination of $m^{u}$ and $m^{s}$.
\end{proof}

\begin{lemma}\label{l.igualdade}
$\lambda_+(\sA,\sps)=\int \phi_\sA \, dm^u$ and $\lambda_-(\sA,\sps)=\int \phi_\sA \, dm^s$.
\end{lemma}

\begin{proof}
Let $v^u_\bx$ be a unit vector in the Oseledets subspace $E^u_\bx$. Then
\begin{align*}
\lambda_+(\sA,\bx)
& = \lim_{n\rightarrow \infty}\frac{1}{n}\log\| A^n(\bx)v^u_\bx\|
  = \lim_{n\rightarrow \infty}\frac{1}{n}\sum_{j=0}^{n-1}\log\| A(f^j(\bx))v^u_{f^j(\bx)}\| \\
& = \lim_{n\rightarrow \infty}\frac{1}{n}\sum_{j=0}^{n-1}\phi_\sA(f^j(\bx), E^u_{f^j(\bx)})
  = \tilde{\phi}_\sA(\bx,E^u_\bx)
\end{align*}
for $\mu$-almost every $\bx$, where $\tilde{\phi}_\sA$ is the Birkhoff average of $\phi_\sA$
for $F_\sA$. Hence,
\begin{align*}
\lambda_+(\sA,\sps)
& = \int \tilde{\phi}_\sA(\bx, E^u_\bx)\,d\mu(\bx)
% = \int (\int \tilde{\phi}_\sA \, d \delta_{(\bx,E^u_\bx)})\,d\mu \\ &
  = \int \tilde{\phi}_\sA \, dm^u
  = \int \phi_\sA \, dm^u.
\end{align*}
Analogously, $\lambda_-(\sA,\sps)=\int \phi_\sA dm^s$. This completes the proof.
\end{proof}

\begin{remark}\label{r.uniqueustaterealizing}
It follows from Lemma~\ref{l.combinacaolinear} that $m^{u}$ is the unique invariant measure
$m$ such that $\lambda_+(\sA,\sps)=\int \phi_\sA \, dm$.
\end{remark}

\subsubsection{Stationary measures}
Given $(\sB,\sqs)$ in $\cS(X) \times \cP(\cX)$, a probability $\eta$ on $\Proj(\CC^2)$
is called $(\sB,\sqs)$-\emph{stationary} if
\begin{equation}\label{eq.stationary}
\eta = \int (B_x)_* \eta \,d\sqs(x).
\end{equation}
The next lemma asserts that the stationary measures are the projections to
$\Proj(\CC^2)$ of the $u$-states of the corresponding cocycle.
%The projection $\eta$ to $\Proj(\CC^{2})$ of a probability measure $m$ in
%$M\times\Proj(\CC^{2})$ is defined by $\eta(E) = m(M\times E)$ for every
%measurable $E\subset\Proj(\CC^{2})$.
We are going to denote $M^u=\cX^{\Z_+}$ and $M^s=\cX^{\Z_-}$.
Notice that $\sqs^\Z=\mu^s\times\mu^u$ where $\mu^*$ is a measure on $M^*$,
for $*\in\{s, u\}$.

\begin{lemma}\label{l.uestadoseestacionaria}
If $m$ is an invariant $u$-state for $(\sB,\sqs)$ then its projection $\eta$ to $\Proj(\CC^2)$ is a
$(\sB,\sqs)$-stationary measure. Conversely, given any $(\sB,\sqs)$-stationary $\eta$ there exists an
invariant $u$-state that projects to $\eta$.
\end{lemma}

\begin{proof}
Let $\bx\mapsto m_{\bx}$ be a disintegration of $m$ constant along unstable leaves.
For any measurable set $I\subset\Proj(\CC^{2})$,
$$
\eta(I)=m(M\times I) = \int m_{\bx}(M\times I) \, d\mu(\bx)   = \int m_{f(\bx)}(M\times I) \, d\mu(\bx) \\
$$
because $\mu$ is $f$-invariant. Since $m$ is $F_{\sB}$-invariant,
the expression on the right hand side may be rewritten as
%$$
\begin{align*}
\int \sB(\bx)_* m_{\bx}(M\times I) & \, d\mu(\bx) \\
    & = \int_{M^{s}} \big(\int_{M^{u}} \sB(\bx)_* m_{\bx}(M\times I) \, d\mu^{u}(\bx^{u})\big) \,d\mu^{s}(\bx^{s}).
\end{align*}
%$$
Since the disintegration is constant on local unstable sets and $B(\bx^{s},\bx^{u})$ depends
only on $\bx^s$ (we write $B(\bx^{s})$ instead), this last expression coincides with
$$
\begin{aligned}
& \int_{\Sigma^{s}}\sB(\bx^{s})_* \big(\int_{\Sigma^{u}} m_{\bx^{u}}(M\times I) \, d\mu^{u}(\bx^{u})\big) \,d\mu^{s}(\bx^{s}) \\
%\sum_{i\in\cI} q_{i} \int (B_i)_* m_{\bx}(M\times I) \, d\mu(\bx)
%= \sum_{i\in\cI} p_{i} \, (f \times B_i)_* \int  m_{\bx}(G) \, d\mu(\bx)
%& = \sum_{i\in\cI} q_{i} \, (B_i)_*m(M\times I) \\
%& = \sum_{i\in\cI} q_{i} \, (B_i)_*\eta(I).
& = \int_{\Sigma^{s}}\sB(\bx^{s})_*\eta(I) \,d\mu^{s}(\bx^{s})
 = \int \sB(\bx)_*\eta(I) \,d\mu(\bx)
 = \int (B_x)_*\eta(I) \,d\sqs(x).
\end{aligned}
$$
Thus, $\eta=\int(\sB_x)_*\eta \, d\sqs(x)$ as claimed.

Conversely, given any $(\sB,\sqs)$-stationary measure $\eta$, consider the sequence
of functions
$$
m^n: \bx \mapsto m^{n}_\bx=\sB^{n}(f^{-n}(\bx))_*\eta
$$
with values in the space of probabilities on $\Proj(\CC^2)$. It is clear from the
definition that each $m^n$ is measurable with respect to the $\sigma$-algebra
$\cF^n$ of subsets of $M$ generated by the cylinders
$$
[-n:\,\Delta_{-n},\dots,\Delta_{-1}]=
\{\bx\in M: x_i \in \Delta_i \text{ for } i=-n,\dots,-1\},
$$
where the $\Delta_i$ are measurable subsets of $\cX$. These $\sigma$-algebras $\cF^n$ form
a non-decreasing sequence. We claim that $(m^n,\cF^n)$ is a martingale, that is,
\begin{equation}\label{eq.martingale}
\int_{C} m^{n+1} \, d\mu =\int_{C} m^{n}\, d\mu
\quad\text{for every $C\in\cF^n$ and every $n\ge 1$.}
\end{equation}
To prove this, it suffices to treat the case when $C$ is a cylinder $[-n:\Delta_{-n},\dots,\Delta_{-1}]$.
Then, for any $n\ge 1$,
$$
\begin{aligned}
\int_C \sA^{n+1}(f^{-n-1}(\bx))_*\eta \,d\mu(\bx)
& = \int_C \sA^{n}(f^{-n}(\bx))_*\sA(f^{-n-1}(\bx))_*\eta\, d\mu(\bx)\\
%& = \int_{\Delta\times \cX}(A_{x_{-1}}\dots A_{x_{-n}})_*(A_{x_{-n-1}})_*\eta \,d\sps(x_{-n-1})\dots d\sps(x_{-1}) \\
& = \int_{C}\sA^{n}(f^{-n}(\bx))_*\big[\int_\cX(A_{y})_*\eta\, dp(y)\big]d\mu(\bx)\\
%& = \int_{\Delta}(A_{x_{-1}}\dots A_{x_{-n}})_*\eta\,\sps(x_{-n})\dots d\sps(x_{-1})\\
& = \int_C \sA^n(f^{-n}(\bx))_*\eta\, d\mu
\end{aligned}
$$
because $\eta$ is stationary. This proves the claim \eqref{eq.martingale}.
Then, by the martingale convergence theorem
(see \cite[Chapter~5]{Br68}), there exists a function $\bx \mapsto m_\bx$ such that $m^n_\bx$
converges $\mu$-almost everywhere to $m_\bx$ in the weak$^*$ topology.
Let $m$ be the probability measure defined on $M\times\Proj(\CC^2)$ by
$$
m(E) = \int m_\bx\big(E \cap (\{\bx\}\times\Proj(\CC^2))\big)\,d\mu(\bx)
$$
for any measurable set $E$. By construction, the disintegration $\bx \mapsto m_\bx$
is constant on every $\{\bx^s\}\times M^u$. This means that $m$
is a $u$-state. Also by construction, $m_{f(\bx)} = A(\bx)_* m_\bx$ for $\mu$-almost
every $\bx\in M$. This proves that the $u$-state $m$ is invariant. Moreover,
by \eqref{eq.martingale} and the assumption that $\eta$ is stationary,
$$
m^n(M\times I)
%= \int_M m^n_\bx(I)\,d\mu(\bx)
%= \int_M m^1_\bx(I)\,d\mu(\bx)
= m^1(M\times I)
%= \int_M \sA^n(\bx)_*\eta(I)\,d\mu(x)
= \int_M (A_x)_*\eta(I)\,dp(x)
= \eta(I)
$$
for every $n\ge 1$ and any measurable set $I\subset \Proj(\CC^2)$. This means that
$m^n$ projects to $\eta$ for every $n\ge 1$. Then so does the limit $m$.
This completes the proof of the lemma.
\end{proof}

We are also going to show that the projection of $m^u$ to the projective space $\Proj(\C^2)$
completely determines the Lyapunov exponents:

\begin{lemma}\label{l.etadeterminaexpoente}
%Suppose $\lambda_+(\sA,\sps)>0$ and let $\eta$ be the projection of $m^u$ to $\Proj(\CC^2)$.
Let $m$ be a $u$-state realizing $\lambda_+(\sA,\sps)$ and let $\eta$ be its projection to
$\Proj(\C^2)$. Then
$$
\lambda_+(\sA,\sps)
=\int \int \log \phi_A(x,v)\,d\eta([v])\,d\sps(x)
%=\int \int \log \frac{\|A_xv\|}{\|v\|}\,d\eta([v])\,d\sps(x).
$$
\end{lemma}
\begin{proof}
Suppose first that $\lambda_+(\sA,\sps)=0$. By Lemmas~\ref{l.ledrappier} and \ref{l.suestados},
every $F_\sA$-invariant probability $m$ that projects down to $\mu$ realizes the largest exponent
and is a product measure $m=\mu\times\eta$. Thus, in this case, the lemma follows immediately from
Fubini's Theorem.
If $\lambda_+(\sA,\sps)>0$, then $m^{u}$ is the unique $u$-state that realizes $\lambda_+$.
Then a straightforward calculation,
\begin{align*}
\lambda_+(\sA,\sps)&=\int_{M}\log\|\sA(\bx)E_\bx^{u}\|d\mu %\\
%&
=\int_{M^{s}}\int_{M^{u}}\log\|A(\bx^{s})E_{\bx^u}^{u}\|d\mu(\bx^{u})\,d\mu^{s}\\
&=\int_{\cX}\int_{M^{u}}\log\|A_yE_{\bx^u}^{u}\|d\mu(\bx^{u})\,d\sps(y)\\
&=\int_{\cX}\int_{M^{u}}\int_{\Proj(\C^{2})}\log\frac{\|A_y v\|}{\|v\|}\,d\delta_{E_{x^u}^{u}}d\mu(\bx^{u})\,d\sps(y)\\
&=\int_{\cX}\int_{\Proj(\C^{2})}\log\frac{\|A_y v\|}{\|v\|}\,d\eta([v])\,d\sps(y),
\end{align*}
concludes the proof of the lemma.
\end{proof}

\begin{lemma}\label{l.stationaryisclose}
If $(\sA^k,\sps^k)_k$ converges to $(\sA,\sps)$ and $\eta^k$ is a sequence of
$(\sA^k,\sps^k)$-stationary measure converging to $\eta$ then $\eta$ is
an $(\sA,\sps)$-stationary measure.
\end{lemma}
\begin{proof}
We have to show that
$$
\lim_k \int (A^k_x)_*\eta^k \,d\sps^k=\int (A_x)_*\eta \,d\sps
$$
in the weak$^*$ sense. Let $\phi:\Proj(\C^2)\to \R$ be a continuous function. Then
$$
|\int \int \phi(A^k_xv)\,d\eta^k\,d\sps^k - \int \int \phi(A_xv)\,d\eta\,d\sps|\le a_k + b_k + c_k
$$
where
$$
\begin{aligned}
&a_k=|\int \int \phi(A^k_xv)\,d\eta^k\,d\sps^k - \int \int \phi(A_xv)\,d\eta^k\,d\sps^k|
\\&b_k=|\int \int \phi(A_xv)\,d\eta^k\,d\sps^k - \int \int \phi(A_xv)\,d\eta\,d\sps^k|
\\&c_k=|\int \int \phi(A_xv)\,d\eta\,d\sps^k - \int \int \phi(A_xv)\,d\eta\,d\sps|
\end{aligned}
$$
It is clear that $(a_k)_k$ converges to zero, because $\|A^k_x-A_x\|$ converges uniformly to zero
and $\phi$ is uniformly continuous. To prove that $b_k$ converges to zero we argue as follows.
Given $\vep>0$, fix $\delta >0$ such that $|\phi(v)-\phi(w)|<\vep/3$ for all $v,w\in\Proj(\C^2)$
such that $d(v,w)<\delta$. Since the image of $\sA$ is contained in a compact subset of $\SL(2,\C)$,
there are $B_1,\dots, B_n\in \SL(2,\C)$ such that their $\delta$-neighborhoods cover $\sA(\cX)$.
The assumption that $(\eta^k)_k$ converges to $\eta$ in the weak$^*$ topology implies that there
exists $k_0\in\N$ such that
$$
|\int \phi(B_iv)\,d\eta^k - \int \phi(B_iv)\,d\eta|<\vep/3
$$
for all $k>k_0$ and for all $i=1,\dots,n$. Then we can use the triangle inequality to conclude that
$$
|\int \phi(A_xv)\,d\eta^k - \int \phi(A_xv)\,d\eta| \le \vep
$$
for all $k>k_0$.
%So, for each $x\in\cX$, take $i=i(x)$ such that $\|A_x-B_i\|<\delta$
%$$
%\begin{aligned}
%b_k(x)&=|\int \phi(A_xv)\,d\eta^k - \int \phi(A_xv)\,d\eta|
%\le \int |\phi(A_xv)-\phi(B_iv)|\,d\eta^k+\\
%& +|\int \phi(B_iv)\,d\eta^k-\int \phi(B_iv)\,d\eta|+
%\int |\phi(A_xv)-\phi(B_iv)|\,d\eta\\
%&\le \vep/3+\vep/3+\vep/3=\vep
%\end{aligned}
%$$
Integrating with respect to $\sps^k$ we conclude that $b_k\le\vep$ for all $k>k_0$.
This proves that $b_k$ converges to $0$. Finally, it is clear that $a_k$ converges
to zero, because our assumptions imply that $(\sps^k)_k$ converges strongly to $\sps$.
The proof of the lemma is complete.
\end{proof}

\subsubsection{Proof of Proposition~\ref{p.naodiagonal}}

Notice that $\lambda_+$ is non-negative and, as observed in \eqref{eq.inf1}--\eqref{eq.inf2},
\begin{equation}\label{eq.semi}
(\sA,\sps) \mapsto \lambda_+(\sA,\sps)
 = \inf_n \frac 1n \int \log\|A^n(\bx)\|\,d\mu(\bx)
\end{equation}
is upper-semicontinuous for the topology defined %by \eqref{eq.topology1} and, \emph{a fortiori},
by \eqref{eq.topology2}.
So, if $(\sA,\sps)\in\cS(\cX)\times\cP(\cX)$ is a discontinuity point for the largest
Lyapunov exponent then $\lambda_+(\sA,\sps)>0$ and there is a sequence $(\sA^k,\sps^k)_k$
converging to $(\sA,\sps)$ as $k\to\infty$ such that
$$
\lim_k\lambda_+(\sA^k,\sps^k)<\lambda_+(\sA,\sps).
$$
As we have seen, for each $k$ there exists some $(\sA^k,\sps^k)$-stationary measure $\eta^k$
satisfying
$$
\int_\cX \int_{\Proj(\C^2)}\log\|A_x^kv\|\,d\eta^k(v)d\sps^k(x) = \lambda_+(\sA^k,\sps^k).
$$
Up to restricting to a subsequence, we may assume that $(\eta^k)_k$ converges in the weak$^*$
topology to some probability measure $\eta$ on $\Proj(\CC^2)$.
Then $\eta$ is an $(\sA,\sps)$-stationary measure, by Lemma~\ref{l.stationaryisclose}.
Using Lemma~\ref{l.etadeterminaexpoente} we see that
$$\begin{aligned}
\int_\cX \int_{\Proj(\C^2)}\log\|A_xv\|\,d\eta(&v)d\sps(x)
= \lim_k\int_\cX \int_{\Proj(\C^2)}\log\|A^k_xv\|\,d\eta^k(v)d\sps^k(x)
\\& < \lambda_+(\sA,\sps)
 = \int_\cX \int_{\Proj(\C^2)}\log\|A_xv\|\,d\eta^u(v)d\sps(x)
\end{aligned}
$$
where $\eta^u$ is the projection of $m^u$. In particular, by Lemma~\ref{l.uestadoseestacionaria},
there exists an invariant $u$-state $m \neq m^u$. It follows, using
Lemma~\ref{l.combinacaolinear}, that
$$
m=\alpha m^u +\beta m^s \quad\text{with $\alpha+\beta=1$ and $\beta\neq 0$.}
$$
This implies that $m^s$ is a $u$-state, because it is a linear combination
of $m$ and $m^u$.
Hence $m^s$ is an $su$-state. In view of Lemma~\ref{l.suestados} this means that the
Oseledets subspace $E_\bx^s$ is constant on a full $\mu$-measure set.
Let $F^s\in\Proj(\CC^2)$ denote this constant.
Analogously, using that $(\sA,\sps)$ is a discontinuity point for the smallest Lyapunov
exponent, we find $F^u\in\Proj(\CC^2)$ such that $E^u_\bx=F^u$ for $\mu$-almost
every $\bx$. It is clear that $F^u$ and $F^s$ are both invariant under $A_x$,
for $\sps$-almost every $x\in\cX$, because $\mu=\sps^{\Z}$. This means that there
exists $\cZ\subset\cX$ with $\sps(\cZ)=1$ such that the linear operators defined by
the $A_y$, $y\in \cZ$ have a common eigenbasis, which is precisely the first claim in
the proposition. The last claim (commutativity) is a trivial consequence.
This completes the proof of Proposition~\ref{p.naodiagonal}.

\subsection{Handling the diagonal case}\label{ss.proofdiagonal}

Here we prove Proposition~\ref{p.diagonal}.
Let $(\sA,\sps)\in\cS(X) \times \cP(\cX)$ and $\cZ$ be as in the conclusion of Proposition~\ref{p.naodiagonal}
and consider any $\sps\in\cP(\cX)$.
Since conjugacies preserve the Lyapunov exponents, we may suppose $P=\id$ and
\begin{equation}\label{eq.diagonal}
A_x=\left(
            \begin{array}{cc}
               \theta_x & 0 \\
              0 & \theta_x^{-1} \\
            \end{array}
          \right)
\quad \text{for all}\quad x\in\cZ.
\end{equation}
Notice that the Lyapunov exponents of $(\sA,\sps)$ are
\begin{equation}\label{eq.expoentes}
\pm \int_{\cZ} \log |\theta_x|\,d\sps(x).
\end{equation}
If they vanish then $(\sA,\sps)$ is automatically a continuity point, and so there is
nothing to prove. Otherwise, it is no restriction to suppose
\begin{equation}\label{eq.sinal}
\int_{\cZ}\log |\theta_x| > 0.
\end{equation}

Let $V_\vep$ be the $\vep$-neighborhood of the horizontal direction in
$\Proj(\CC^2)$ and $\cZ$ be as given in Proposition~\ref{p.naodiagonal}.
The key step in the proof of Theorem~\ref{t.Bernoulli} is the following

\begin{proposition}\label{keyprop}
Given $\vep >0$ and $\delta >0$ there exists $\gamma >0$ such that if
$(\sB,\sqs) \in V(\sA,\sps,\gamma,\cZ)$ and there is no one-dimensional
subspace invariant under all $B_x$ for $x$ in a full $\sqs$-measure then
$\eta(V_\vep^{c}) \le \delta$
for any $(\sB,\sqs)$-stationary measure $\eta$.
\end{proposition}

The proof of Proposition~\ref{keyprop} will be given in Section~\ref{s.keyproof}.
Right now, let us conclude the proof of Proposition~\ref{p.diagonal}.

Let $(\sB,\sqs)\in\cS(\cX)\times\cP(\cX)$ be close to $(\sA,\sps)$ in the sense of
\eqref{eq.topology2}.
First, suppose there exists some one-dimensional subspace $r\subset\CC^2$ invariant
under all the $B_x$, $x$ in a $\sqs$-full measure. Then $r$ must be close to either the vertical axis
or the  horizontal axis: that is because \eqref{eq.sinal} implies $|\theta_{x}|\neq 1$
for some $\sqs$-positive measure subset.
Then the Lyapunov  exponent of $(\sB,\sqs)$ along $r$ is close to one of the exponents
\eqref{eq.expoentes}. Since the other exponent is symmetric, this proves that the
Lyapunov exponents of $(\sB,\sqs)$ are close to the Lyapunov exponents of $(\sA,\sps)$.
Now assume $\sB$ does not admit any invariant one-dimensional subspace.
Let $M>0$ such that $M^{-1}\|v\|<\|B_xv\|<M\|v\|$ for $\sps$-almost every $x\in \cX$,
all $v\in \C^{2}$ and $d(\sA,\sB)<1$. Let $0 \ll \vep \ll \delta \ll \rho \ll 1$.
Let $m$ be any $u$-state realizing the largest Lyapunov exponent of $(\sB,\sqs)$,
and $\eta$ its projection on $P(\C^{2})$. By Proposition~\ref{keyprop},
$$
\begin{aligned}
\int_{\Proj(\C^2)}\log\frac{\|B_xv\|}{\|v\|}d\eta([v])&=
\int_{V_\vep^{c}}\log\frac{\|B_xv\|}{\|v\|}d\eta([v])
+\int_{V_\vep}\log\frac{\|B_xv\|}{\|v\|}d\eta([v])
\\& \ge -\delta \log M + \eta(V_\vep)(\log
|\theta_x|-\delta)
\end{aligned}
$$
for $\sqs$-almost every $x\in\cX$. Together with Lemma~\ref{l.etadeterminaexpoente}, this implies
$$
\lambda_+(\sB,\sqs)> \eta(V_\vep)\lambda_+(\sA,\sps)-\delta (\log M + \eta(V_\vep))> \lambda_+(\sA,\sps) - \rho.
$$
Upper semi-continuity gives $\lambda_{+}(\sB,\sqs) \le \lambda_+(\sA,\sps)+\rho$.
Thus, we have shown that $(\sA,\sps)$ is indeed a continuity point for the Lyapunov
exponents.

This reduces the proof of Proposition~\ref{p.diagonal} and Theorem~\ref{t.Bernoulli} to proving
Proposition~\ref{keyprop}.

\section{Proof of the Key Proposition}\label{s.keyproof}

Here we give a suitable reformulation of Proposition~\ref{keyprop} and reduce its proof to two
technical estimates, Propositions~\ref{p.oito} and~\ref{p.oitenta}, whose proof will be presented
in the next section.

\subsection{Preliminary observations}

As a first step we note that under the assumptions of the proposition all stationary
measures are non-atomic.

\begin{lemma}\label{l.nonatomicmeasure}
There exists $\gamma>0$ such that if $(\sB,\sqs)\in V(\sA,\sps,\gamma,\cZ)$ and there
is no one-dimensional subspace of $\RR^2$ invariant under $B_x$ for every $x$ in a full
$\sqs$-measure, then every $(\sB,\sqs)$-stationary measure is non-atomic.
\end{lemma}

\begin{proof}
By assumption, $\sA$ is diagonal and the Lyapunov exponents do not vanish. So, we may take
$\gamma>0$ so that if $(\sB,\sqs)\in V(\sA,\sps,\gamma,\cZ)$ then $B_x$ is
hyperbolic and its eigenspaces are close to the horizontal and vertical directions,
for every $x$ in some set $\cL\subset\cX$ with $\sqs(\cL)>0$.
Then any finite set of one-dimensional subspaces invariant under any $B_x$, $x\in\cL$ has
at most two elements. Moreover, they must coincide with the eigenspaces of $B_x$ and,
consequently, are actually fixed under $B_x$. Since we assume there is no one-dimensional
subspace fixed by $B_x$ for $\mu$-almost every $x$, it follows that there is no finite set of
one-dimensional subspaces invariant under $B_x$ for $\mu$-almost every $x$.

Now let us suppose $\eta$ has some atom. Let $z_1$, \dots, $z_N$ be the atoms with the
largest mass, say, $\eta(\{z_i\})=a$ for $i=1$, \dots, $N$. Since $\eta$ is a stationary
measure,
$$
\eta\big(\{B_x^{-1}(z_1),\dots,B_x^{-1}(z_N)\}\big)=\eta\big(\{z_1,\dots,z_N\}\big)=Na
$$
for $\sqs$-almost every $x\in\cX$. Moreover, in view of the previous paragraph, we have
$\{B_x^{-1}(z_1),\dots,B_x^{-1}(z_N)\} \neq \{z_1, \dots,z_N\}$ for a positive $\sqs$-measure
subset of points $x$. This implies that there exists $z\neq z_i$ for $i=1,\dots,N$ such
that $\eta(\{z\})=a$.  That contradicts the choice of the $z_i$ and so the lemma
is proved.
\end{proof}

Let $\phi:\Proj(\CC^2)\rightarrow \CC^2\cup\{\infty\}$, $\phi([z_1,z_2])=z_1/z_2$
be the standard identification between the complex projective space and the
Riemann sphere. Then the projective action of a linear map
$$
B=\left(\begin{array}{cc}
                 a & b \\
                 c & d \\
          \end{array}
  \right)
$$
corresponds to a M\"obius transformation on the sphere
$$
\hat{B}:\C\cup\{\infty\}\rightarrow \C\cup\{\infty\} \quad \hat{B}(z)=\frac{az+b}{cz+d},
$$
in the sense that $\phi\circ B =\hat{B}\circ\phi$. It follows that a measure $\xi$ in
projective space is $(\sB,\sqs)$-stationary if and only if the measure $\eta=\phi_*\xi$
on the sphere satisfies $\eta=\int (\hat{B}_x)_*\eta \,d\sqs(x)$.
Then the measure $\eta$ is also said to be $(\sB,\sqs)$-stationary. Clearly, $\eta$ is
non-atomic if and only if $\xi$ is.

This means that the key Proposition~\ref{keyprop} may be restated as

\begin{proposition}\label{keyprop2}
Given $\vep >0$ and $\delta >0$ there exist $\gamma >0$ such that if
$(\sB,\sqs)\in V(\sA,\sps,\gamma,\cZ)$ and $\sqs(\{x\in\cX: \hat{B}_x(z)=z\})<1$
for all $z\in\C \cup \{\infty\}$ then
\[
\eta(B(0,\vep^{-1})) \le \delta
\]
for any $(\sB,\sqs)$-stationary probability measure $\eta$ on $\C\cup \{\infty\}$.
\end{proposition}

The proof of this proposition will appear in the next section. Let us briefly comment on
the statement and the overall strategy of the proof. As mentioned before, the set
$\Stat(A,p)$ of stationary measures varies in a semi-continuous fashion with the data:
if $(B,q)$ is close to $(A,p)$ then every $(B,q)$-stationary measure is close to
$\Stat(A,p)$. This is not sufficient for our purposes because in the diagonal case
there are several stationary measures, not all of which realize the largest Lyapunov
exponent. Indeed, the assumption that both the vertical direction and the horizontal
direction are invariant under almost every $A_x$ implies that both associated Dirac masses
on the Riemann sphere, $\delta_0$ and $\delta_\infty$, are $(A,p)$-stationary measures,
and so $\Stat(A,p)$ the whole line segment between these two Dirac masses.

To establish continuity of the Lyapunov exponents we need to prove the much finer fact
that stationary measures of nearby (irreducible) cocycles are close to the one element of
$\Stat(\sA,\sps)$, namely $\delta_\infty$, that realizes the Lyapunov exponent
$\lambda_+(\sA,\sps)$. That is the meaning of the key proposition. The reason we may
restrict ourselves to irreducible cocycles is because in the reducible case continuity
follows from a different, and much easier argument, as we have seen.

The crucial property that singles out $\delta_{\infty}$ among all $(\sA,\sps)$-stationary
measures is the fact that it is an attractor for the random walk defined by $(A,p)$ on
$\Proj(\CC^2)$. Indeed, the random trajectory $A^n(x)\xi$ of any $\xi\in\Proj(\CC^2)\setminus\{0\}$
converges to $\infty$ almost surely. Consequently, the forward iterates of any probability
$\eta$ with $\eta(\{0\})=0$ under the dynamics
\begin{equation}\label{eq.funcaoestac}
f_\sA: \eta \mapsto \int (A_x)_*\eta \,d\sps(x)
\end{equation}
induced by $\sA$ in the space of the probability measures of $\Proj(\C^2)$ converge to
$\delta_\infty$.

The heart of the proof is, thus, a robustness theorem for certain random walks.
We prove that the attractor persists for all nearby irreducible cocycles: if $(\sB,\sqs)$ is close
enough to $(\sA,\sps)$ and there is no one-dimensional subspace invariant under $\sqs$-almost
every $B_x$, then $f_\sB$ possesses an attractor that is strongly concentrated near $\infty$,
and draws the forward iterates of every Dirac mass. In particular, every fixed point $\eta$
of the operator $f_\sB$ must be strongly concentrated near $\infty$, as claimed.

While the details are fairly lengthy, the main ideas in the proof are very natural,
so that applications of this approach to much more general situations can be expected.
In particular, there is some promising progress in the setting of H\"older continuous
(not locally constant) two-dimensional cocycles over hyperbolic systems.

\subsection{Auxiliary statements}

Recall, from \eqref{eq.diagonal} and \eqref{eq.sinal}, that
\begin{equation}\label{eq.Ai}
A_x=\left(
             \begin{array}{cc}
               \theta_x & 0 \\
               0 & \theta_x^{-1} \\
             \end{array}\right) \quad\text{with}\quad \int \log|\theta_x|d\sps(x) > 0
\end{equation}
for every $x\in\cZ$. By definition, $\sqs(\cZ)=1$ for all $(\sB,\sqs)\in V(\sA,\sps,\gamma,\cZ)$.
Thus, up to  restricting all cocycles to a full measure subset, which does not affect the Lyapunov
exponents, we may assume that $\cZ=\cX$. We do so in all that follows.
Let $\sB$, $\sqs$, and $\eta$ be as in the statement.

\begin{lemma}\label{l.lemmakey1}
There are $\beta, \sigma \in (0,1)$, $k\in\N$, positive numbers $(\sigma_x)_{x\in\cX}$,
and integers $(s_x)_{x\in\cX}$ such that
\begin{itemize}
\item[(a)] $0 <\|\sA\|^{-1}/4\le \sigma_x \le \beta |\theta_x|$ for all $x\in\cX$
\item[(b)] $\sigma_x=\sigma^{s_x}$ for all $x\in\cX$
\item[(c)] $\int \log\sigma_x\, d\sps(x) > 4/k$.
\end{itemize}
\end{lemma}

\begin{proof}
Fix $k\in\N$ large enough so that $\int \log|\theta_x|\,d\sps(x) > 7 / k$.
Define $\log\beta = \log\sigma = -1/k$. For each $x\in \cX$, define
$$
r_x = \big[k\log|\theta_x|\big],
\quad
s_x = \left\{\begin{array}{ll} r_x-1  & \text{if $r_x\neq 1$ } \\
r_x-2 & \text{if $r_x=1$}\end{array}\right.
\quad
\log\sigma_x = -\frac{s_x}{k}.
$$
Properties (a) and (b) follow immediately. Moreover,
$$
\int \log\sigma_x \, d\sps(x)\ge \int \big(\log|\theta_x|-3/k\big)\, d\sps(x)> 4/k
$$
as claimed in (c). The proof is complete.
\end{proof}

Let $\sigma$, $\beta$, $\sigma_x$, and $s_x$, be as in Lemma~\ref{l.lemmakey1}.
We partition $\cX = \cX_- \cup \cX_+$, where $\cX_-$ is the subset of $x\in\cX$ with
$s_x<0$ (i.e. $\sigma_x>1$) and $\cX_+$ is the subset of $x\in\cX$ with $s_x>0$
(i.e. $\sigma_x<1$). For each $x\in\cX$, let
\begin{equation}\label{eq.Dx}
D_x=\left(
             \begin{array}{cc}
               \sigma_x & 0 \\
               0 & \sigma_x^{-1} \\
             \end{array}\right)
\quand \hat{D}_x(z)=\sigma_x^{2}z.
\end{equation}
Consider also
\begin{equation}\label{eq.Dsp}
D_{sp}=\left(
             \begin{array}{cc}
               \sigma^{\tau} & 0 \\
               0 & \sigma^{-\tau} \\
             \end{array}\right)
\quand \hat{D}_{sp}(z)=\sigma^{2\tau}z,
\end{equation}
where $\tau$ is the smallest integer such that $\sigma^{\tau}\le \|A\|^{-1}/4$.
Given any $\cK\subset \cX$, let $\sK$ be the cocycle defined by
%\begin{equation}\label{eq.Kx}
%\sK_x = \left\{\begin{array}{ll} D_x  & \text{if $x\in\cK$ } \\
%D_{sp} & \text{if $x\in\cX\setminus\cK$}\end{array}\right.
%\end{equation}
\begin{equation}\label{eq.Kx}
\sK_x=\left(
             \begin{array}{cc}
               k_x & 0 \\
               0 & k_x^{-1} \\
             \end{array}\right)
\quad \text{where}\quad
k_x = \left\{\begin{array}{ll} \sigma_x  & \text{if $x\in\cK$ } \\
\sigma^{\tau} & \text{if $x\in\cX\setminus\cK$.}\end{array}\right.
\end{equation}

\begin{lemma}\label{l.alpha}
There exist $\alpha>0$ and $\tilde\alpha>0$ such that, given any measurable set $\cK\subset\cX$
with $\sps(\cK)\ge 1-\alpha$,
$$
\int \log k_x \,d\sps(x)\ge 2/k
%\int_{\cK} \log\sigma_x\,d\sps(x)\ge 2/k -\log\sigma^{\tau}\sps(\cX\setminus \cK)
\quand\sps\big(\{x: k_x>1\}\big)\ge\tilde\alpha.
$$
\end{lemma}
\begin{proof}
Taking $\alpha=(-k\log\sigma^{\tau})^{-1}$, we have
$$
\begin{aligned}
\int \log k_x\,d\sps
& \ge \int \log\sigma_x\,d\sps + \int_{\cX\setminus \cK}\big(\log\sigma^{\tau}-\log\sigma_x\big) \,d\sps\\
&\ge 4/k +2\log \sigma^{\tau}\sps(\cX\setminus \cK) \ge 2/k.
\end{aligned}
$$
This proves the first claim. The second one is a direct consequence, with $\tilde\alpha=2/(k\sup k_x)$.
\end{proof}

For $z_0\in\C$ and $r\ge0$, we denote $B(z_0,r)=\{z\in\C: |z-z_0|\leq r\}$.
Given $\sB$, $\sC\in\cS(X)$ and $\cY\subset\cX$ we say that $r\ge0$ is \emph{$(\sB,\cY)$-centered}
with respect to $\sC$ if
\begin{equation}\label{eq.contractive}
\hat{B}_x^{-1}(B(0,r)) \subset \hat{C}_x^{-1}(B(0,r))
\quad\text{for every $x\in\cY$.}
\end{equation}
When $\cY=\cX$ we just say that $r$ is \emph{$\sB$-centered} with respect to $\sC$.
Given $\sB$, $\sC\in\cS(\cX)$, $\sqs\in\cP(\cX)$, and a $(\sB,\sqs)$-stationary measure $\eta$,
we say that $r\ge0$ is \emph{$(\sB,\sqs,\eta)$-targeted} with respect to $\sC$ if
\begin{equation}
\int \eta\big(\hat{B}_x^{-1}(B(0,r))\big)\,d\sqs(x)\le \int \eta\big(\hat{C}_x^{-1}(B(0,r))\big)\,d\sqs(x)
\end{equation}

\begin{remark}
If $r\ge0$ is $\sB$-centered (respectively, $(\sB,\sqs,\eta)$-targeted) with respect to $\sD$
then it is also $\sB$-centered (respectively, $(\sB,\sqs,\eta)$-targeted) with respect to the
cocycle $\sK$ defined in \eqref{eq.Kx}. That is because
$\hat{D}_x^{-1}(B(0,r))\subset\hat{D}_{sp}^{-1}(B(0,r))$ for any $x\in\cX$.
\end{remark}

The following simple facts will be useful in what follows:

\begin{lemma}\label{l.seis-a}
Given $\rho>0$ there is $\gamma>0$ such every $r\in[\rho,\rho^{-1}]$
is $\sB$-centered with respect to $\sD$ for every $\sB\in\cS(X)$ with
$d(\sA,\sB)<\gamma$.
\end{lemma}

\begin{proof}
By assumption, $\pm\log|\theta_x|$, $x\in\cX$ is bounded.
Write
$$
B_x^{-1}=\left(\begin{array}{cc} a_x & b_x \\ c_x & d_x\end{array}\right).
$$
The condition $d(\sA,\sB)<\gamma$ implies that $|a_x-\theta_x^{-1}|$,
$|b_x|$, $|c_x|$, $|d_x-\theta_x|$ are all less than $c_1 \gamma$ for
some constant $c_1$ independent of $x$ and $\gamma$.
Given $\rho>0$, assume first that $\gamma\le\rho^2$.
Then, for any $|z|\in[\rho,\rho^{-1}]$,
$$
|\hat B_x^{-1}(z)|
 % = \frac{|a_x z + b_x|}{|c_x z + d_x|}
  \le \frac{|a_x z| + |b_x|}{|d_x| - |c_x z|}
  \le \frac{|a_x| + c_1\sqrt\gamma}{|d_x| - c_1 \sqrt\gamma}|z|
  \le \frac{|\theta_x^{-1}|}{|\theta_x|} \frac{1+ c_2 \sqrt\gamma}{1 - c_2 \sqrt\gamma}|z|
$$
where $c_2$ is also independent of $x$ and $\gamma$.
Thus, there exists $\gamma_0>0$, independent of $x\in\cX$ such that,
if $d(\sA,\sB)<\gamma\le \gamma_0$ then
$$
|\hat B_x^{-1}(z)|
\le (\beta |\theta_x|)^{-2} |z|
\le \sigma_x^{-2} |z|
= |\hat D_x^{-1}(z)|
$$
for every $x\in\cX$ and $|z|\in[\rho,\rho^{-1}]$.
This gives that every $r\in[\rho,\rho^{-1}]$ is $\sB$-centered with respect to $\sD$, as claimed.
\end{proof}

\begin{lemma}\label{l.proporcional}
There are $\gamma>0$ and $c>0$ such that if $d(\sA,\sB)<\gamma$ and $x\in X_0$ is such
that $\hat{B}_x$ has a fixed point in $B(0,\rho)$, for some $\rho<c^{-1}$, then every
$r\in[c\rho,1]$ is $(B,\{x\})$-centered with respect to $D$.
%%% MARCELO: mudei a redacao do enunciado, mas o conteudo deve ter ficado o mesmo
%There are $c>0$ and $\gamma>0$ such that if $\rho<c^{-1}$, $x\in X_0$, $d(\sA,\sB)<\gamma$,
%and $\hat{B}_x$ has some fixed point in $B(0,\rho)$, then
%\begin{equation}
%\hat{B}_x^{-1}(B(0,r))\subset\hat{D}_x^{-1}(B(0,r)) \quad \text{for all } r\in[c\rho,1].
%\end{equation}
\end{lemma}

\begin{proof}
First, take $\gamma>0$ such that $d(A,B)<\gamma$ implies that $\hat{B}_x^{-1}$ is a
$\lambda_x$-contraction with
$$
\frac{\|\sA\|^{-1}}{2}=\lambda\le\lambda_x \le (1+\rho_0)^{-2}
$$
and $\hat{D}_x^{-1}(z)=\Lambda_x z$ with $\lambda_x \le \beta\Lambda_x$.
Then choose $c>0$ large enough so that $\lambda^{-1} < c(\beta^{-1}-1-c^{-1})$.
It follows that
$$
|\hat{B}_x^{-1}(z)|\le [c^{-1}+\lambda_x(1+c^{-1})]r \le \Lambda_x r
$$
whenever $|z|\le r$ and $r\in[c\rho,1]$. In other words,
$$
\hat{B}_x^{-1}(B(0,r))\subset\hat{D}_x^{-1}(B(0,r)) \quad \text{for all } r\in[c\rho,1],
$$
as claimed. This proves the lemma.
\end{proof}

The proof of Proposition~\ref{keyprop2} relies on a couple of technical results, Propositions~\ref{p.oito}
and~\ref{p.oitenta}, that we state in the sequel and whose proofs will appear in Section~\ref{s.estimates}.
The first proposition gives a bound on the mass of the stationary measure away from the vertical
(and the horizontal) direction. Fix $\alpha>0$ as in Lemma~\ref{l.alpha}, once and for all.

\begin{proposition}\label{p.oito}
Given $\vep>0$ and $\delta>0$ there exists $\gamma>0$ such that if $d(\sA,\sB)<\gamma$
and $d(\sps,\sqs)<\gamma$ then
$$
\eta\big(B(0,\vep^{-1}) \setminus B(0,r_0)\big) \le \delta
$$
for any $(\sB,\sqs)$-stationary measure $\eta$ and any $r_0\in(0,1)$ such that every
$r\in [r_0,\vep^{-1}]$ is $(\sB,\cK)$-centered with respect to $\sD$ for some measurable
set $\cK$ with $\sps(\cK)\ge 1-\alpha$.
\end{proposition}

What we actually use is the following consequence:

\begin{corollary}\label{c.oito}
Given $\vep>0$ and $\delta>0$ there exist $\gamma>0$ such that if $d(\sA,\sB)<\gamma$
and $d(\sps,\sqs)<\gamma$ then either
$
\eta\big(B(0,\vep^{-1})\big) \le \delta
$
or there exists $r_0 \in (0,1)$ such that
$$
\eta\big(B(0,\vep^{-1}) \setminus B(0,r_0)\big) \le \delta
$$
and $\sps(\{x\in \cX: \hat{B}_x^{-1}(B(0,r_0))\nsubset \hat{D}_x^{-1}(B(0,r_0)) \})\ge \alpha$.
\end{corollary}

\begin{proof}
Let $r_1\ge 0$ be the infimum of all $r\in(0,1)$ such that
$$
\eta\big(B(0,\vep^{-1})\setminus B(0,r)\big) < \delta.
$$
If $r_1=0$ then $\eta\big(B(0,\vep^{-1})\setminus \{0\}\big) \le \delta$.
Since $\eta$ has no atoms, by Lemma~\ref{l.nonatomicmeasure}, it follows that
$\eta(B(0,\vep^{-1})) \le \delta$. This proves the corollary in this case.
Now, suppose $r_1>0$. Then, $\eta(B(0,\vep^{-1})\setminus B(0,r_1)) \ge \delta$,
and so, by Proposition~\ref{p.oito},
$$
A_1=\{x\in \cX: \hat{B}_x^{-1}(B(0,r_1))\nsubset\hat{D}_x^{-1}(B(0,r_1)) \}
$$
has $\sps(A_1)>\alpha$. Let $1>r_2>r_3>\dots$ be a decreasing sequence converging to $r_1$, and
$$
A_k=\{x\in \cX: \hat{B}_x^{-1}(B(0,r_k))\nsubset\hat{D}_x^{-1}(B(0,r_k)) \},
$$
for $k=2, 3, \dots$. Notice that $\liminf_k A_k \supset A_1$ and so, by the Lemma of Fatou,
$\liminf_k\sps(A_k) \ge \sps(A_1)>\alpha$.
%The indicator functions satisfy
%$$
%\liminf_k \cX_{A_k} \ge \cX_{A_1}
%$$
%and so, by the Lemma of Fatou,
%$$
%\liminf \int \cX_{A_k}\,d\sps \ge \int\liminf \cX_{A_k}\,d\sps
%\ge \int \cX_{A_1}\,d\sps.
%$$
In particular, there is $N\ge 2$ such that $\sps(A_N))\ge \alpha$.
The proof is complete, taking $r_0=r_N$.
\end{proof}

Our second technical proposition will allow us to bound the mass of the stationary measure
close to the vertical direction:

\begin{proposition}\label{p.oitenta}
There are $\gamma>0$ and $N\ge 1$ such that if $d(\sA,\sB)<\gamma$ and $r_0\in[0,1]$
and $x\in\cX$ are such that $\hat{B}_x^{-1}(B(0,r_0))\nsubset \hat{D}_x^{-1}(B(0,r_0))$,
then
$$
\cD \cap \hat{B}_x^{-1} (\cD)=\emptyset,\quad\text{where }
 \cD = \left\{\begin{array}{ll} \hat{D}_x^{- N}(B(0,r_0)) & \text{if } x \in \cX_- \\
                                \hat{D}_x^{N}(B(0,r_0)) & \text{if } x \in \cX_+.
                   \end{array}\right.
$$
In particular, $B(0,\sigma^{2N\tau}r_0) \cap \hat{B}_x^{-1} (B(0,\sigma^{2N\tau}r_0))=\emptyset$.
\end{proposition}

\subsection{Proof of Proposition \ref{keyprop2}}

The assumption $\lambda_+(\sA,\sps)>0$ implies that there exist $\alpha_0>0$ and $\rho_0>0$
such that
$$
X_0=\{x\in\cX: |\theta_x|>1+\rho_0\}
$$
has $\sps(X_0)\ge \alpha_0$. Let $c>0$ and $N\ge 1$ be fixed as in Lemma~\ref{l.proporcional}
and Proposition~\ref{p.oitenta}, respectively.
Denote $\beta_0={2\alpha_0}/{(1+8c^2\sigma^{-4\tau N})}$.
For each $z\in \C$ and $\rho\in[0,1)$, define
$$
\Gamma(z,\rho)=\{x\in X_0: \text{$\hat{B}_x$ has some fixed point in $B(z,\rho)$}\}.
$$
In particular, $\Gamma(z,0)$ is the set of $x\in X_0$ such that $z$ is fixed under $\hat B_x$.
Observe also that $\Gamma(z,0)=\cap_{\rho>0}\Gamma(z,\rho)$.

\begin{lemma}\label{l.aaa}
Given $\rho>0$ there exist $\gamma>0$ and $\lambda_0 \in(0,1)$ such that if $d(\sA,\sB)<\gamma$
then $\hat{B}_x^{-1}(B(0,r))\subset B(0,\lambda_0 r)$ for every $1>r \ge c\rho$ and every
$x\in \Gamma(0,\rho)$. In particular, there is $\kappa\ge 1$ such that
$\hat{B}_x^{-\kappa}(B(0,r))\subset B(0,\sigma^{2\tau}r)$ for $1>r\ge c\rho\sigma^{-2\tau}$
and $x\in \Gamma(0,\rho)$.
\end{lemma}

\begin{proof}
Take $\gamma>0$ small enough to ensure that every $\hat{B}_x^{-1}$, $x\in X_0$ is a contraction
on the ball $B(0,1)$, with uniform contraction rate $\lambda\in(0,1)$. Then, consider
$\lambda_0=\lambda(1+c^{-1}) + c^{-1}$. Fix $x\in X_0$ and let $z_0\in B(0,1)$ be the unique
fixed point of $\hat{B}_x^{-1}$. For any $z$ with $|z|=r \ge c\rho$,
$$
|\hat{B}_x^{-1}(z)|
\le \rho + \lambda|z-z_0|
\le \rho + \lambda(r+\rho)
\le [c^{-1}+\lambda(1+c^{-1})]r
\le \lambda_0 r.
$$
This proves the first claim in the statement. To get the second statement, just take $\kappa\ge 1$
to be the smallest positive integer such that $\lambda_0^{\kappa} \le \sigma^{2\tau}$.
\end{proof}

We distinguish two cases in the proof of the proposition. First, we take the cocycle to be
``reducible'', in the sense that the $B_x$ have a common invariant line, for a subset of
values of $x\in X_0$ with sizable mass. More precisely, we suppose that
\begin{equation}\label{eq.reducible}
\sps(\Gamma(z_0,0))\ge \beta_0
\quad\text{for some } z_0\in B(0,1).
\end{equation}
It is no restriction to suppose that $z_0=0$, as we will see in a while, so let us do that for
the time being. Then, \eqref{eq.reducible} implies that $\sqs(\Gamma(0,0))\ge \beta_0/2$ for every
$\sqs$ in a neighborhood of $\sps$.
Suppose, by contradiction, that $\eta\big(B(0,\vep^{-1})\big) > \delta$.
Then, by Corollary~\ref{c.oito}, there exists $r_0 \in (0, 1)$ such that
\begin{equation}\label{eq.awayx}
\eta\big(B(0,\vep^{-1})\setminus B(0,r_0)\big)\le \delta \quand \sps(Y) \ge \alpha,
\end{equation}
where $Y=\{x\in \cX: \hat{B}_x^{-1}(B(0,r_0))\nsubset \hat{D}_x^{-1}(B(0,r_0)) \}$.
The latter implies that $\sqs(Y)\ge \alpha/2$ for every $\sqs$ sufficiently close to $\sps$.
Lemma~\ref{l.aaa} implies that
$$
\begin{aligned}
\sqs(\Gamma(0,0)) & \eta\Big(B(0,r_0) \setminus B(0,\lambda_0 r_0)\Big) \\
& = \int_{\Gamma(0,0)}\eta\big(B(0,r_0)\big) - \eta\big(B(0,\lambda_0r_0)\big)\,d\sqs(x) \\
& \le \int_{\Gamma(0,0)}\Big(\eta(B(0,r_0)) - \eta(\hat{B}_x^{-1}(B(0,r_0)))\Big)\,d\sqs(x)
%& \le \int_{\Gamma(0,0)}\eta\Big(B(0,r_0)\setminus \hat{B}_x^{-1}(B(0,r_0))\Big)\,d\sqs(x) \\
%& = \int_{\Gamma(0,0)}\Big(\eta(B(0,r_0)) - \eta(\hat{B}_x^{-1}(B(0,r_0)))\Big)\,d\sqs(x)
\end{aligned}
$$
Since $\eta$ is stationary, the last expression coincides with
$$
\int_{\cX\setminus \Gamma(0,0)} \Big(\eta(\hat{B}_x^{-1}(B(0,r_0)))-\eta(B(0,r_0))\Big)\,d\sqs(x),
$$
which is, clearly, bounded above by $\eta\big(B(0,\vep^{-1})\setminus B(0,r_0)\big)$.
In this way, using \eqref{eq.awayx}, we find that
$$
\sqs(\Gamma(0,0)) \eta\big(B(0,r_0)\setminus B(0,\lambda_0 r_0)\big)
\le \eta\big(B(0,\vep^{-1})\setminus B(0,r_0)\big) \le \delta.
$$
Recall that $q(\Gamma(0,0))\ge \beta_0/2$. Then, using \eqref{eq.awayx} once more,
$$
\eta\big(B(0,\vep^{-1})\setminus B(0,\lambda_0 r_0)\big)
\le \delta + 2 \delta\beta_0^{-1}.
$$
Arguing by induction we get that
$$
%\eta\big(B(0,r_0) \setminus \hat{B}_\ell^{-j}(B(0,r_0))\big) \le
\eta\big(B(0,\vep^{-1}) \setminus B(0,\lambda_0^{j}r_0)\big)
\le \delta (1+2\beta_0^{-1})^{j} \quad\text{for every $j\ge 0$.}
$$
In particular, this holds for $j=\kappa N$.
%\begin{equation}\label{eq.1}
%\eta\big(B(0,\vep^{-1}) \setminus (B(0,\lambda_0^{\kappa N}r_0)\big)
%\leq \delta (1+8\beta_0^{-1})^{\kappa N}.
%\end{equation}
Hence, cf. Lemma~\ref{l.aaa},
\begin{equation}\label{eq.x}
\eta\big(B(0,\vep^{-1}) \setminus B(0,\sigma^{2\tau N}r_0)\big)
\le \delta (1+2\beta_0^{-1})^{\kappa N}.
\end{equation}
Denote $\cB_0=B(0,\sigma^{2\tau N}r_0)$. From Proposition~\ref{p.oitenta}
we get that $\cB_0$ and its pre-image under $\hat{B}_x$ are disjoint
for every $x\in Y$. So, \eqref{eq.x} implies
\begin{equation}\label{eq.xx}
\eta\big(\hat{B}_x^{-1}(\cB_0)\big)\le \delta (1+2\beta_0^{-1})^{\kappa N}
\quad\text{for every $x\in Y$.}
\end{equation}
Since $\eta$ is stationary,
$$
\begin{aligned}
\sqs(Y)\eta(\cB_0)
& = \int_{\cX\setminus Y}\eta\big(\hat{B}_x^{-1}(\cB_0) \setminus \cB_0\big)\,d\sqs(x)
 + \int_Y \eta\big(\hat{B}_x^{-1}(\cB_0)\big)\,d\sqs(x) \\
& \le \int_{\cX\setminus Y}\eta\big(B(0,\vep^{-1}) \setminus \cB_0\big)\,d\sqs(x)
 + \int_Y \eta\big(\hat{B}_x^{-1}(\cB_0)\big)\,d\sqs(x)
\end{aligned}
$$
Recall that $q(Y)\ge\alpha/2$. Hence, using \eqref{eq.x} and~\eqref{eq.xx},
\begin{equation}\label{eq.xxx}
\eta(\cB_0) \le 4 \delta \alpha^{-1} (1+2\beta_0^{-1})^{\kappa N}.
\end{equation}
Adding \eqref{eq.x} and~\eqref{eq.xxx} we conclude that
\begin{equation}\label{eq.tildec1}
\eta(B(0,\vep^{-1})) \le \tilde c \delta,
\quad\quad\quad \tilde c = (1+4\alpha^{-1})(1+2\beta_0^{-1})^{\kappa N}.
\end{equation}

So far we have been assuming that the fixed point sits at $z_0=0$. Let us now explain
how this assumption can be removed. Notice that for every $x\in X_0$ the matrix $A_x$
is diagonal, its larger eigenvalue is far from the unit circle, and the corresponding
eigenvector is horizontal.
Thus, an attracting fixed point $z_0\in B(0,1)$ as in \eqref{eq.reducible} must be
close to zero (in other words, the direction it represents is close to horizontal)
if the cocycle $\sB$ is close to $\sA$. Define
$$
H=\left(
      \begin{array}{cc}
        a_0 & -b_0 \\
        b_0 & a_0 \\
      \end{array}
    \right),
$$
where $(a_0, b_0) \approx (1, 0)$ be a unit vector in the direction represented
by $z_0$, and then consider the cocycle $\sC$ defined by $C_x=H\,B_x\,H^{-1}$.
Clearly, $\hat B_x^{-1}(z_0)=z_0$ translates to $\hat C_x^{-1}(0)=z_0$.
Moreover, if $\eta$ is $(\sB,\sqs)$-stationary then $H_*\eta$ is
$(\sC,\sqs)$-stationary. Thus, we can use the arguments in the previous paragraph
to conclude that
$$
H_*\eta(B(0,2\vep^{-1})) \le \tilde c \delta.
$$
Finally, $H(B(0,\vep^{-1})) \subset B(0,2\vep^{-1})$ because $H$ is close to the
identity, and so it follows that
\begin{equation}\label{eq.tildec2}
\eta(B(0,\vep^{-1})) \le \tilde c \delta
\end{equation}
also in this case. One can easily dispose of the factor $\tilde c$.
So, the proof of Proposition~\ref{keyprop2} in the reducible case is complete.

\medskip

Now, we assume that the cocycle is ``irreducible'', in the sense that
$\sps(\Gamma(z,0))<\beta_0$ for all $z\in B(0,1)$. We need the following lemma:

\begin{lemma}\label{l.zzero}
There exists $\gamma>0$ such that if $d(\sA,\sB)<\gamma$ and $\sps(\Gamma(z,0))<\beta_0$
for all $z\in B(0,1)$ then for each (small) $\varsigma>0$ there exist $z_0\in B(0,1)$
and $\rho_0>0$ such that
\begin{itemize}
  \item [(a)] $\sps(\Gamma(z,\rho_0)) \le \sps(\Gamma(z_0,\rho_0))+\varsigma$ for all $z\in B(0,1)$;
  \item [(b)] $\beta_0/4 \le \sps(\Gamma(z_0,\rho_0))\le \beta_0$;
  \item [(c)] $\sps(X_0\setminus \Gamma(z_0,c\sigma^{-2\tau N}\rho_0)) \ge \beta_0/2$.
\end{itemize}
\end{lemma}

\begin{proof}
Let $\varrho=\inf\{r>0: \sps(\Gamma(z,r)) > \beta_0 \text{ for some } z\in B(0,1)\}$.
We claim that $\varrho>0$. Indeed, suppose that for each $n\in \N$ there exists
$z_n\in B(0,1)$ such that $\sps(\Gamma(z_n,1/n)) > \beta_0$. We may suppose that $(z_n)_n$
converges to some $\tilde z \in B(0,1)$. Then $\sps(\Gamma(\tilde z,r)) > \beta_0$ for
any $r>0$, and so $\sps(\Gamma(\tilde z,0))\ge \beta_0$. The latter contradicts the hypothesis,
and so our claim is proved. Now, define $\rho_0 = {9\varrho}/{10}$ and let
$$
S=\sup\{\sps(\Gamma(z,\rho_0)): z\in B(0,1)\}.
$$
Notice that $S\le\beta_0$, because $\rho_0<\varrho$. We claim that $S > \beta_0/4$.
Indeed, by the definition of $\varrho$, one may find $z\in B(0,1)$ such that
$p(\Gamma(z,11\varrho/10)) > \beta_0$. It is easy to check that $\Gamma(z,11\varrho/10)$
may be covered with not more than four sets $\sps(\Gamma(z',\rho_0))$, $z'\in B(0,1)$.
Then, $p(\Gamma(z',\rho_0))>\beta_0/4$ for some choice of $z'$, and that proves the claim.
Now, given any small $\varsigma>0$, take $z_0\in B(0,1)$ such that
$p(\Gamma(z_0,\rho_0))+\varsigma> S$. Properties (a) and (b) follow immediately from
the previous considerations. We are left to prove (c). Clearly, one can find $G\subset \CC$
with $\# G \le 4 c^2 \sigma^{-4 \tau N}$ such that $\{\Gamma(z,\rho_0): z \in G\}$ covers
$\Gamma(z_0,c\sigma^{-2 \tau N}\rho_0))$. Consequently, since the supremum $S \le \beta_0$.
$$
\begin{aligned}
\mu\big(X_0\setminus \Gamma(z_0,c\sigma^{-2 \tau N}\rho_0))\big)
& \ge p(X_0) - \sum_{z\in G} \mu(\Gamma(z,\rho_0)) \\
& \ge \alpha_0 - 4 c^2 \sigma^{-4 \tau N} \beta_0,
\end{aligned}
$$
Now notice that $\beta_0$ was defined in such a way that this last expression is equal to
$\beta_0/2$. This completes the proof of the lemma.
\end{proof}

Let $z_0$ and $\rho_0>0$ be as given by Lemma~\ref{l.zzero}, for some sufficiently small
$\varsigma>0$. For the same reasons as in the reducible case, it is no restriction to
suppose that $z_0=0$. Define $X_1=X_0\setminus\Gamma(0,c\sigma^{-2\tau N}\rho_0)$.
By parts (c) and (d) of Lemma~\ref{l.zzero},
\begin{equation}\label{eq.X1}
\sps(\Gamma(0,\rho_0)) \ge \beta_0/4 \quand \sps(X_1) \ge \beta_0/2.
\end{equation}
Suppose, by contradiction, that $\eta\big(B(0,\vep^{-1})\big) > \delta$.
Then take $r_0\in(0,1)$ as in Corollary~\ref{c.oito}:
\begin{equation}\label{eq.Y0}
\eta\big(B(0,\vep^{-1})\setminus B(0,r_0)\big)\le \delta \quand \sps(Y_0) \ge \alpha,
\end{equation}
where $Y_0=\{x\in \cX: \hat{B}_x^{-1}(B(0,r_0))\nsubset \hat{D}_x^{-1}(B(0,r_0)) \}$.
Let $r_1=\max\{r_0,c\sigma^{-2\tau N}\rho_0\}$ and
$$
Y_1=\{x\in \cX: \hat{B}_x^{-1}(B(0,r_1))\nsubset\hat{D}_x^{-1}(B(0,r_1)) \}.
$$
Lemma~\ref{l.zzero} and \eqref{eq.Y0} imply
\begin{equation}\label{eq.Y1}
\eta\big(B(0,\vep^{-1})\setminus B(0,r_1)\big)\le\delta
\quand p(Y_1)\ge \beta_1
\end{equation}
where $\beta_1=\min\{\alpha,\beta_0/2\}$. Indeed, the first claim in \eqref{eq.Y1}
is a direct consequence of \eqref{eq.Y0}. For the second claim there are two cases.
If $r_1=r_0$ then $Y_1=Y_0$ and so \eqref{eq.Y0} yields $p(Y_1)\ge \alpha$.
If $r_1 = c\sigma^{-2\tau N}\rho_0$ we may use Lemma~\ref{l.zzero}(c), together with
the observation that
$$
Y_1^c
 = \{x\in \cX: \hat{B}_x^{-1}(B(0,r_1))\subset\hat{D}_x^{-1}(B(0,r_1))\}
 \subset \Gamma(0,r_1),
$$
to conclude that $p(Y_1)\ge\beta_0/2$. This establishes \eqref{eq.Y1}. It follows
that
$$
\sqs(\Gamma(0,\rho))>\beta_0/8 \quand \sqs(Y_1)>\beta_1/2,
$$
as long as $d(\sps,\sqs)$ is sufficiently small. Lemma~\ref{l.aaa} implies that
$$
\begin{aligned}
\sqs(\Gamma(0,\rho_0)) & \eta\Big(B(0,r_1) \setminus B(0,\lambda_0 r_1)\Big) \\
& = \int_{\Gamma(0,\rho_0)}\Big(\eta(B(0,r_1)) -\eta(B(0,\lambda_0 r_1))\Big)\,d\sqs(x) \\
& = \int_{\Gamma(0,\rho_0)}\Big(\eta(B(0,r_1)) - \eta(\hat{B}_x^{-1}(B(0,r_1)))\Big)\,d\sqs(x)
%& \le \int_{\Gamma(0,\rho)}\eta\Big(B(0,r_1)\setminus \hat{B}_x^{-1}(B(0,r_1))\Big)\,d\sqs(x) \\
%& = \int_{\Gamma(0,\rho)}\Big(\eta(B(0,r_1)) - \eta(\hat{B}_x^{-1}(B(0,r_1)))\Big)\,d\sqs(x)
\end{aligned}
$$
Since $\eta$ is stationary, the last expression coincides with
$$
\int_{\cX\setminus \Gamma(0,\rho_0)} \Big(\eta(\hat{B}_x^{-1}(B(0,r_1)))-\eta(B(0,r_1))\Big)\,d\sqs(x),
$$
which is, clearly, bounded above by $\eta\big(B(0,\vep^{-1})\setminus B(0,r_1)\big)$.
Using \eqref{eq.Y1}, it follows that
$$
\sqs(\Gamma(0,\rho)) \eta\Big(B(0,r_1)\setminus B(0,\lambda_0 r_1)\Big)
\le\delta.
$$
Then, using \eqref{eq.Y1} once more,
\begin{equation}\label{eq.away2}
\eta\big(B(0,\vep^{-1})\setminus B(0,\lambda_0 r_1)\big)
\le \delta(1+8\beta_0^{-1})
\end{equation}
Arguing by induction we get that
$$
%\eta\big(B(0,r_0) \setminus \hat{B}_\ell^{-j}(B(0,r_0))\big) \le
\eta\big(B(0,\vep^{-1}) \setminus B(0,\lambda_0^{j}r_1)\big)
\le \delta (1+8\beta_0^{-1})^{j} \quad\text{for every $j\ge 0$}
$$
(the cases $j=0$ and $j=1$ are given by \eqref{eq.Y1} and \eqref{eq.away2}, respectively).
In particular, this holds for $j=\kappa N$. Hence, cf. Lemma~\ref{l.aaa},
\begin{equation}\label{eq.xbis}
\eta\big(B(0,\vep^{-1}) \setminus B(0,\sigma^{2\tau N}r_1)\big)
\le \delta (1+8\beta_0^{-1})^{\kappa N}.
\end{equation}
Denote $\cB_1=B(0,\sigma^{2\tau N}r_1)$. From Proposition~\ref{p.oitenta}
we get that $\cB_1$ and its pre-image under $\hat{B}_x$ are disjoint
for every $x\in Y_1$. So, \eqref{eq.xbis} implies
\begin{equation}\label{eq.xxbis}
\eta\big(\hat{B}_x^{-1}(\cB_1)\big)\le \delta (1+8\beta_0^{-1})^{\kappa N}
\quad\text{for every $x\in Y_1$.}
\end{equation}
Since $\eta$ is stationary,
$$
\begin{aligned}
\sqs(Y_1)\eta(\cB_1)
& = \int_{\cX\setminus Y_1}\eta\big(\hat{B}_x^{-1}(\cB_1) \setminus \cB_1\big)\,d\sqs(x)
 + \int_{Y_1} \eta\big(\hat{B}_x^{-1}(\cB_1)\big)\,d\sqs(x) \\
& \le \int_{\cX\setminus Y_1}\eta\big(B(0,\vep^{-1}) \setminus \cB_1\big)\,d\sqs(x)
 + \int_{Y_1} \eta\big(\hat{B}_x^{-1}(\cB_1)\big)\,d\sqs(x)
\end{aligned}
$$
Combining $q(Y_1)\ge\beta_1/2$ with \eqref{eq.xbis} and~\eqref{eq.xxbis},
we find that
\begin{equation}\label{eq.xxxbis}
\eta(\cB_1) \le 4 \delta \beta_1^{-1} (1+8\beta_0^{-1})^{\kappa N}.
\end{equation}
Adding \eqref{eq.xbis} and~\eqref{eq.xxxbis} we conclude that
\begin{equation}\label{eq.tildec3}
\eta(B(0,\vep^{-1})) \le \tilde c \delta,
\quad\quad\quad \tilde c = (1+4\beta_1^{-1})(1+8\beta_0^{-1})^{\kappa N}.
\end{equation}
That completes the proof in the irreducible case, under the assumption that $z_0=0$.
This assumption can be removed in just the same way as before in the reducible case,
and so our argument is complete.

\section{Main estimates}\label{s.estimates}

All we have to do to finish the proof of Proposition~\ref{keyprop2} is to prove
Propositions~\ref{p.oito} and~\ref{p.oitenta}.

\subsection{Mass away from the vertical}\label{ss.away}

In this section we prepare the proof of Proposition~\ref{p.oito}.
Let $\sigma<1$ be as in Lemma~\ref{l.lemmakey1}.
For each $\cK\subset \cX$ consider the associated cocycle $\sK$, as defined in \eqref{eq.Kx}.
Clearly,
\begin{equation}\label{eq.MDi2}
\hat{K_x}(B(0,r\sigma^{2j}))=B(0,r\sigma^{2j+2s_x})
\end{equation}
% In particular, the family $\cB(r)=\{B(0,r\sigma^{2j}): j\in\Z)\}$
%is invariant under every $\hat{D}_i$, $i\in\cX$.
for every $r>0$, $x\in\cX$, and $j\in\Z$. Define
\begin{equation}\label{eq.Ij}
I_j(r)= B(0,r\sigma^{2j-2})\setminus B(0,r\sigma^{2j})
\end{equation}
for $j\in\Z$ and
\begin{equation}\label{eq.Li2}\begin{aligned}
L_x(r) = \left\{\begin{array}{ll} B(0,r) \setminus \hat{K_x}^{-1}(B(0,r)) & \text{for $x\in\cX_-$} \\
                                  \hat{K_x}^{-1}(B(0,r)) \setminus B(0,r) & \text{for $x\in\cX_+$}
                \end{array} \right.
\end{aligned}\end{equation}
where $\cX=\cX_-\cup \cX_+$ denotes the partition associated to the cocycle $\sK$,
that is, such that $k_x >1$ for $x\in\cX_-$ and $k_x<1$ for $x\in\cX_+$. Notice that
$\cX\setminus\cK\subset \cX_+$ because $k_x=\sigma^{\tau}$ for all $x\in\cX\setminus\cK$.

\begin{lemma}\label{l.consequencias}
If $r>0$ is $(\sB,\sqs,\eta)$-targeted with respect to $\sK$ then
\begin{enumerate}
\item $\int_{\cX_-}\eta(L_x(r))\,d\sqs(x) \leq \int_{\cX_+}\eta(L_x(r))\,d\sqs(x)$
\smallskip
\item $\int_{\cX_-}\sum_{j=1}^{-s_x} \eta(I_j(r))\,d\sqs(x)
\leq \int_{\cX_+}\sum_{j=-s_x+1}^{0} \eta(I_j(r))\,d\sqs(x)$.
\end{enumerate}
More generally, given $n\ge 0$, if $r \sigma^{2t}$ is $(\sB,\sqs,\eta)$-targeted
with respect to $\sK$ for $t = 0, \dots, n$, then
\[
\int_{\cX_-}\sum_{j=t+1}^{t-s_x} \eta(I_j(r))\,d\sqs(x)
\leq \int_{\cX_+}\sum_{j=t-s_x+1}^{t} \eta(I_j(r))\,d\sqs(x)
\text{ for } t=0, \dots, n.
\]
\end{lemma}

\begin{proof}
Denote $J=B(0,r)$. Using that $r$ is $(\sB,\sqs,\eta)$-centered and $\eta$ is $(\sB,\sqs)$-stationary
$$
\int \left(\eta(J) - \eta(\hat{K_x}^{-1}(J))\right)\,d\sqs(x) \le
\int \left(\eta(J) - \eta(\hat{B}_x^{-1}(J))\right)\,d\sqs(x) = 0.
$$
By the definition \eqref{eq.Li2}, the left hand side coincides with
$$
\int_{\cX_-} \eta (L_x(r))\,d\sqs(x) - \int_{\cX_+}\eta (L_x(r))\,d\sqs(x).
$$
This proves the first claim. The second one is a direct consequence:
just note that, by \eqref{eq.MDi2},
$$
L_x(r)
  = \left\{\begin{array}{ll} B(0,r)\setminus B(0,r\sigma^{-2 s_x}) = \bigsqcup_{j=1}^{-s_x} I_j(r) &
                             \text{ for $x \in\cX_-$} \\
                             B(0,r\sigma^{-2 s_x})\setminus B(0,r) = \bigsqcup_{j=-s_x+1}^{0} I_j(r) &
                             \text{ for $x\in\cX_+$.} \end{array}\right.
$$
The last claim follows, noticing
$I_j(r\sigma^{2t})=I_{j+t}(r)$ for all $j$ and $r$.
\end{proof}

Let $\alpha$ and $\gamma$ be the constants in Lemma~\ref{l.alpha}.

\begin{corollary}\label{c.consequencias}
If $\sps(\cK)\ge 1-\alpha$ and $r>0$ is $(\sB,\sqs,\eta)$-targeted with respect to $K$ then
$$
\eta(I_1(r))\le \tilde\alpha^{-1}\sum_{j=-n_++1}^{0} \eta(I_j(r)),
\quad\text{where } n_+=\sup_{x\in\cX_+}|s_x|.
$$
\end{corollary}

\begin{proof}
Lemma~\ref{l.alpha} gives that $\sps(\cX_-) \ge \tilde\alpha$. Part (2) of Lemma~\ref{l.consequencias}
implies
$$
\sps(\cX_-) \eta(I_1(r)) \le \sps(\cX_+) \sum_{j=-n_x+1}^0\eta(I_j(r))
$$
The conclusion of the corollary follows, immediately.
\end{proof}

\begin{remark}\label{r.consequencias}
If $r$ is $\sB$-centered then the conclusions of Lemma~\ref{l.consequencias} and
Corollary~\ref{c.consequencias} hold for every $(\sB,\sqs)$-stationary measure $\eta$.
\end{remark}

\begin{lemma}\label{l.oitenta1}
There exists $\gamma > 0$ such that if $d(\sA,\sB)<\gamma$ and $r\in[0,1]$
and $x\in\cX$ are such that $\hat{B}_x^{-1}(B(0,r))\cap B(0,r)\neq\emptyset$, then
$$
\hat{B}_x^{-1}(B(0,r))\cup B(0,r)\subset \hat{D}_{sp}^{-1}(B(0,r)).
$$
\end{lemma}

\begin{proof}
Take $\gamma>0$ such that if $d(\sA,\sB)<\gamma$ then the diameter of $\hat{B}_x^{-1}(B(0,r))$
is less than $3\|\sA\|^2r$, for every $r\in[0,1]$ and $x\in\cX$. Then,
$\hat{B}_x^{-1}(B(0,r))\cap B(0,r)\neq\emptyset$ implies
$$
\hat{B}_x^{-1}(B(0,r))\cup B(0,r)\subset B(0,4\|\sA\|^2r)\subset\hat{D}_{sp}^{-1}(B(0,r)).
$$
This proves the claim.
\end{proof}

We also need the following calculus result. In the application, for proving
Proposition~\ref{p.oito}, we will take $n_x=|s_x|$ and $a_j=\eta(I_j(r))$.

\begin{lemma}\label{l.sete}
Let $(n_x)_{x\in\cX}$ be a bounded family of positive integers and $(a_j)_{j\in\Z}$
be a sequence of non-negative real numbers. Assume that
\begin{itemize}
\item[(a)] $0<S\leq\int_{\cX_-}n_x\,d\sqs(x)-\int_{\cX_+}n_x\,d\sqs(x)$ and
\smallskip
\item[(b)] $\int_{\cX_-}\sum_{j=t+1}^{t+n_x}a_j\,d\sqs(x) \leq \int_{\cX_+} \sum_{j=t-n_x+1}^{t} a_{j}\,d\sqs(x)$
for $t=0, \dots, n$.
\end{itemize}
Denote $n_-=\sup\{n_x: x\in\cX_-\}$ and $n_+=\sup\{n_x: x\in\cX_+\}$.
Then
\[
\sum_{j=1}^{n}a_j\leq \big(\frac{n_- + n_+}{S}\big) \sum_{j=-n_++1}^0 a_j.
\]
\end{lemma}

\begin{proof}
Begin by noting that
\begin{equation}\label{eq.soma1}\begin{aligned}
\sum_{t=0}^{n}\sum_{j=t+1}^{t+n_x} a_j
& = \sum_{l=1}^{n_x} \sum_{j=l}^{n+l} a_j  %\\
%& \ge \sum_{l=1}^{n_i} \left(\sum_{j=1}^{n} a_j - \sum_{j=1}^{l-1} a_j\right)
 \ge n_x \left(\sum_{j=1}^{n} a_j - \sum_{j=1}^{n_x}a_j\right)
\end{aligned}\end{equation}
and that
\begin{equation}\label{eq.soma2}\begin{aligned}
\sum_{t=0}^{n} \sum_{j=t-n_x+1}^{t} a_{j}
& = \sum_{l=-n_x+1}^{0} \sum_{j=l}^{n+l} a_j %\\
%& \le \sum_{l=-n_i+1}^{0} \left(\sum_{j=1}^{n} a_j + \sum_{j=l}^{0} a_j \right)
  \le n_x\left(\sum_{j=1}^{n}a_j + \sum_{j=-n_x+1}^{0}a_j\right)
\end{aligned}\end{equation}
So, adding the inequalities (b) over all $t=0, \dots, n$ and using \eqref{eq.soma1}-\eqref{eq.soma2},
$$
\begin{aligned}
\int_{\cX_-}n_x \left[\sum_{j=1}^n a_j - \sum_{j=1}^{n_x} a_j\right] d\sqs(x)
% & \le \sum_{i=1}^k q_i \sum_{t=0}^n\sum_{j=t+1}^{t+n_i} a _j \\
% & \le \sum_{i=k+1}^d q_i \sum_{t=0}^n\sum_{j=t-n_i+1}^{t} a _j \\
& \le \int_{\cX_+}n_x \left[\sum_{j=1}^n a_j + \sum_{j=-n_x+1}^0 a_j \right] d\sqs(x)
\end{aligned}
$$
or, equivalently,
$$
S \sum_{j=1}^{n} a_j \leq \int_{\cX_-}n_i\sum_{j=1}^{n_x} a_j\,d\sqs(x) +\int_{\cX_+} n_x \sum_{j=-n_x+1}^{0} a_j\,d\sqs(x)
$$
This implies, using the inequality (b) once more,
$$\begin{aligned}
S \sum_{j=1}^{n} a_j
& \leq n_- \int_{\cX_-}\sum_{j=1}^{n_x} a_j\,d\sqs(x) + n_+ \int_{\cX_+} \sum_{j=-n_x+1}^{0} a_j\,d\sqs(x) \\
& \leq (n_- + n_+) \int_{\cX_+}\sum_{j=-n_x+1}^{0} a_j\,d\sqs(x).
\end{aligned}$$
This last expression is bounded above by $(n_- + n_+) \sum_{j=-n_++1}^{0}a_j$.
In this way we get the conclusion of the lemma.
\end{proof}

Define $\alpha_s =\sum_{j=(s-1)n_++1}^{s n_+}a_j$ for each $s \ge 0$.
In the same setting as Lemma~\ref{l.sete}, we obtain

\begin{corollary}\label{c.sete}
Let $n=s_0 n_+$ for some integer $s_0 \geq 1$. There is $s\in\{1,\dots,s_0\}$ such that
$$
\alpha_s
%:=\sum_{j=(s-1)n_++1}^{s n_+}a_j
\leq \big(\frac{n_-+n_+}{s_0S}\big)\alpha_0.
$$
\end{corollary}

\begin{proof}
The conclusion of Lemma~\ref{l.sete} may be rewritten
$$
\sum_{s=1}^{s_0}\alpha_j = \sum_{j=1}^{n} a_j
\le \big(\frac{n_- + n_+}{S}\big) \sum_{j=-n_++1}^0 a_j
= \big(\frac{n_-+n_+}{S}\big)\alpha_0\,.
$$
This implies that $\min_{1\leq s\leq s_0}\alpha_j \le (n_-+n_+)\alpha_0/(s_0 S)$,
as claimed.
\end{proof}

\subsection{Proof of Proposition~\ref{p.oito}}

The claim will follow from applying Lemma~\ref{l.sete} and Corollary~\ref{c.sete} to
appropriate data. As before, let $K$ be the cocycle and $\cX=\cX_+\cup\cX_-$ be the
partition associated to a given set $\cK\subset\cX$.
We break the presentation of the proof into three steps:

\smallskip\paragraph{\bf Step 1:}
Define  $S(\sps)=\int_{\cX_-}n_x\,d\sps(x)-\int_{\cX_+}n_x\,d\sps(x)$, where
$$
n_x
= \frac{|\log k_x|}{|\log\sigma|}
= \left\{\begin{array}{ll} |s_x| & \text{if } x\in \cK \\ \tau & \text{if } x\in\cX\setminus\cK.\end{array}\right.
$$
Let $r_0\ge 0$ be such that every $r\in[r_0,\vep^{-1}]$ is $(\sB,\cK)$-centered with respect to $D$.
Take $\cK$ to have been chose such that $\sps(\cK)\ge 1-\alpha$. Then, by Lemma~\ref{l.alpha},
$$
S(\sps)
 = \int_{\cX} -\frac{\log k_x}{\log \sigma} \, d\sps(x)
 = \int_{\cX} \frac{\log k_x}{|\log \sigma|} \, d\sps(x)
  > \frac{2}{k|\log\sigma|} >0.
$$
Consequently, there exist $\gamma>0$ and $S>0$ such that $S(\sqs)>S$ for every $\sqs$
with $d(\sps,\sqs)<\gamma$. This corresponds to condition (a) in Lemma~\ref{l.sete}.
Given $\vep>0$ and $\delta>0$, let $n = s_0n_+=s_0\tau$ for some integer
$$
s_0 \ge (\frac{n_-+n_+}{S})^2 \delta^{-1}.
$$
Fix also $R > \sigma^{-2n} \vep^{-1}$. By Lemma~\ref{l.seis-a}, there exists $\gamma>0$
such that if $d(\sA,\sB)<\gamma$ then every $r\in[(R\sigma^{-2})^{-1}, R\sigma^{-2}]$ is
$\sB$-centered with respect to $\sD$.
This applies to $y\sigma^{2j}$ for every $j=0, 1, \dots, n$ and any $y\in[R,R\sigma^{-2}]$,
because $y\sigma^{2j} > \vep^{-1} > (R\sigma^{-2})^{-1}$. Fix $y\in[R,R\sigma^{-2}]$
and define $a_j(y)=\eta(I_j(y))$ for $j\in\Z$. Then Lemma~\ref{l.consequencias} gives
\[
\int_{\cX_-}\sum_{j=t+1}^{t+n_x} a_{j}(y)\,d\sqs(x) \leq \int_{\cX_+} \sum_{j=t-n_x+1}^{t} a_{j}(y)\,d\sqs(x)
\]
for all $t=0, \dots, n$. This corresponds to condition (b) in Lemma~\ref{l.sete}.
Thus, we are in a position to apply Corollary~\ref{c.sete}: we conclude that there exists
$s\in\{1,\dots,s_0\}$ such that
\begin{equation}\label{eq.alphas1}
\alpha_{s}(y)
 \le \big(\frac{n_-+n_+}{s_0 S}\big) \alpha_0(y)
 \leq \big(\frac{S}{n_-+n_+}\big) \delta \alpha_0(y)
 \leq \big(\frac{S}{n_-+n_+}\big)\delta
\end{equation}
Notice that, by definition,
$$
\alpha_s(y) = \eta\big(B(0,y\sigma^{2(s-1)n_+})\setminus B(0,y\sigma^{2 s n_+})\big).
$$

\smallskip\paragraph{\bf Step 2:}
Fix $r_1\ge 1$ such that every $r\in [r_1\sigma^{-2},1]$ is $(\sB,\sqs,\eta)$-targeted with
respect to $\sK$ and either $r_1 \le r_0$ or $r_1$ is not $(\sB,\sqs,\eta)$-targeted.
We are going to estimate $\eta(B(0,\vep^{-1}))\setminus B(0,r_1)$,
with the aid of Lemma~\ref{l.sete}.
Condition (a) in the lemma is just the same as before. Concerning condition (b),
notice that $y\in[R,R\sigma^{-2}]$ above may always be chosen so that $r_1=z\sigma^{2\bar{n}}$
for some $\bar{n}\in\N$, where $z=y\sigma^{2sn_+}$.
Then $z\sigma^{2t}$ is $(\sB,\sqs,\eta)$-targeted for every $t=0, 1, \dots, \bar{n}-1$,
due to our choice of $r_1$, and so Lemma~\ref{l.consequencias} gives
\[
\int_{\cX_-}\sum_{j=t+1}^{t+n_x} a_{j}(z)\,d\sqs(x) \leq \int_{\cX_+} \sum_{j=t-n_x+1}^{t} a_{j}(z)\,d\sqs(x)
\]
for all $t=0, \dots, \bar{n}-1$. Thus, applying Lemma~\ref{l.sete},
$$
\sum_{j=1}^{\bar{n}} a_j(z)
\leq \big(\frac{n_- + n_+}{S}\big) \sum_{j=-n_++1}^0 a_j(z).
$$
The left hand side coincides with (recall that $z\ge R\sigma^{2n} >\vep^{-1}$)
$$
\eta\big(B(0,z)\setminus B(0,z\sigma^{2\tilde{n}-2})\big)
\ge \eta\big(B(0,\vep^{-1}) \setminus B(0,r_1\sigma^{-2})\big).
$$
The sum on the right hand side coincides with
$$
\eta\big(B(0,z\sigma^{-2n_+})\setminus B(0,z)\big) = \alpha_0(z)=\alpha_s(y).
$$
Consequently,
\begin{equation}\label{eq.coroabis}
\eta\big(B(0,\vep^{-1}) \setminus B(0,r_1\sigma^{-2})\big) \le \big(\frac{n_- + n_+}{S}\big)\alpha_s(y).
\end{equation}
The relations \eqref{eq.alphas1} and \eqref{eq.coroabis} yield
\begin{equation}\label{eq.coroabis2}
\eta\big(B(0,\vep^{-1}) \setminus B(0,r_1\sigma^{-2})\big) \le \delta.
\end{equation}
Moreover, Corollary~\ref{c.consequencias} gives that
%\begin{equation}\label{eq.coroabis3}
%\eta\big(I_1(r_1\sigma^{-2})\big)
%\le \tilde\alpha^{-1}\sum_{j=-n_++1}^{0}a_j(r_1\sigma^{-2})
%= \tilde\alpha^{-1}\alpha_s(y) \le \tilde\beta^{-1}\delta
%\end{equation}
%where $\tilde\beta=\tilde\alpha(n_-+n_+)/S$.
%
\begin{equation}\label{eq.coroabis3}
\begin{aligned}
\eta\big(I_1(r_1\sigma^{-2})\big)
& \le \tilde\alpha^{-1}\sum_{j=-n_++1}^{0}a_j(r_1\sigma^{-2}) \\
& \le \tilde\alpha^{-1}\eta\big(B(0,\vep^{-1}) \setminus B(0,r_1\sigma^{-2})\big)
\le \tilde\alpha^{-1}\delta
\end{aligned}
\end{equation}
Combining~\eqref{eq.coroabis2} and~\eqref{eq.coroabis3} we obtain
\begin{equation}\label{eq.quaser1}
\eta\big(B(0,\vep^{-1}) \setminus B(0,r_1)\big) \le (1+\tilde\alpha^{-1})\delta.
\end{equation}

\smallskip\paragraph{\bf Step 3:}
If $r_1\le r_0$ then \eqref{eq.quaser1} implies the conclusion of the
proposition (the factor on the right hand side can be avoided replacing $\delta$
by a convenient multiple throughout the argument).
Otherwise, $r_1$ is not $(B,q,\eta)$-targeted with respect to $K$,
and so we must have
\begin{equation}\label{eq.quaser2}
\eta(\hat{B}_x^{-1}B(0,r_1)) > \eta(\hat{K}_x^{-1}B(0,r_1))
\end{equation}
for some $x\in\cX$. Notice that $x$ must belong to $\cX\setminus\cK$, since
$$
\eta(\hat{B}_y^{-1}B(0,r_1))
\le \eta(\hat{D}_{y}^{-1}B(0,r_1))
= \eta(\hat{K}_{y}^{-1}B(0,r_1))
$$
for every $y\in \cK$, because $r_1$ is $(B,\cK)$-centered.
Then \eqref{eq.quaser2} becomes
$$
\eta(\hat{B}_x^{-1}B(0,r_1)) > \eta(\hat{D}_{sp}^{-1}B(0,r_1)).
$$
It follows, using Lemma~\ref{l.oitenta1}, that $B(0,r_1)$ and $\hat{B}_x^{-1}(B(0,r_1))$
are disjoint. By \eqref{eq.quaser1}, this implies that
$$
\eta\big(\hat{B}_x^{-1}(B(0,r_1))\big)\le(1+\tilde\alpha^{-1})\delta.
$$
Moreover,
$\eta(B(0,r_1))\le \eta(\hat{D}_{sp}^{-1}B(0,r_1)) \le \eta(\hat{B}_x^{-1}B(0,r_1))$
and so the previous relation implies that
$$
\eta(B(0,r_1))\le(1+\tilde\alpha^{-1})\delta.
$$
Using \eqref{eq.quaser1} once more, we conclude that
$\eta(B(0,\vep^{-1}))\le2(1+\tilde\alpha^{-1})\delta$.
This implies the conclusion of the proposition (as before, the factor on the right
hand side can be avoided replacing $\delta$ by a convenient multiple throughout
the argument), and so the proof of the proposition is complete.

\subsection{Mass close to the vertical}\label{ss.close}

Next, we are going to prepare the proof of Proposition~\ref{p.oitenta}.
A M\"obius transformation $h$ is a \emph{$\gamma_0$-deformation} of $f(z)=\lambda z$
if $h(z)=(az+b)/(cz+d)$ for some choice of the coefficients satisfying
$$
\max\big\{\big|\frac{|a|}{|\lambda|}-1\big|, \big|b\big|, \big|c\big|, \big||d|-1\big|\big\} < \gamma_0.
$$

\begin{lemma}\label{l.aproximacao}
Given $\beta_0$, $\sigma_0 \in (0,1)$ there are $\gamma_0>0$ and $N_0\in\N$ such that for
any $f(z)=\lambda z$ and $g(z)=\Lambda z$ with $|\lambda|\le\beta_0|\Lambda|$ and
$|\Lambda|\le\sigma_0$, and for any $\gamma_0$-deformation  $\tilde{f}$ of $f$, we have
$$
\tilde{f}(g^{N_0}(B(0,r)))\cap g^{N_0}(B(0,r))=\emptyset.
$$
for any $r\in[0,1]$ such that $\tilde{f}(B(0,r))\nsubset g(B(0,r))$.
\end{lemma}

\begin{proof}
Fix $N_0\in\N$ such that
\begin{equation}\label{eq.N1}
|\Lambda|^{N_0-1}
\le \sigma_0^{N_0-1}
\le \frac{1-\beta_0}{100}
\end{equation}
and $\gamma_0>0$ given by
\begin{equation}\label{eq.gamma1}
\gamma_0 = \frac{1-\beta_0}{100} < \frac{1}{100}
\end{equation}
Write $\tilde{f}(z) = (az+b)/(cz+d)$. If $\tilde{f}(0)=0$ then $b=0$ and
\eqref{eq.gamma1} gives
$$
|\tilde{f}(z)|
%= \frac{|az|}{|cz+d|}
%\le \frac{|a||z|}{|d|-|c||z|}
\le \frac{|az|}{|d|-|c|}
\le \frac{1+\gamma_0}{1-2\gamma_0}|\lambda z|
%\le (1+4\gamma_0)|\lambda z|
\le \beta_0^{-1}|\lambda z|
\le |\Lambda z|
$$
whenever $|z|\le 1$. This means that $\tilde{f}(B(0,r))\subset g(B(0,r))$
for all $r\le 1$, in which case we have nothing to do.
So, let us suppose that $b\neq 0$. Take
$$
r_0 = \frac{10|b|}{|\Lambda|(1-\beta_0)}.
$$
Then $|\tilde{f}(z)| \le |\Lambda z|$ for every $|z| \in [r_0,1]$. Indeed,
$$
|\tilde{f}(z)|
\le \frac{|az|+|b|}{|d|-|c|}
\le \frac{|\lambda|(1+\gamma_0)+|\Lambda|(1-|\beta_0|)/10}{1-2\gamma_0}|z|
$$
and, in view of \eqref{eq.gamma1}, the right hand side is bounded by
$$
\frac{\beta_0 (1 + \gamma_0) + 10\gamma_0}{1-2\gamma_0}|\Lambda z|
\le \frac{\beta_0 + 20\gamma_0}{1-2\gamma_0}|\Lambda z|
\le |\Lambda z|.
$$
This gives that $\tilde{f}(B(0,r))\subset g(B(0,r))$ for every $r\in[r_0,1]$.
Now consider $r\in[0,r_0]$. By \eqref{eq.N1},
$$
|\Lambda|^{N_0} r \le |\Lambda|^{N_0} r_0
\le \frac{|\Lambda|(1-\beta_0)}{100} \frac{10|b|}{|\Lambda|(1-\beta_0)}
\le \frac{|b|}{10} \le \frac{|b|}{5|d|}
$$
and that means that
\begin{equation}\label{eq.umlado}
g^{N_0}(B(0,r))\subset B\big(0,\frac{|b|}{5|d|}\big).
\end{equation}
The relation \eqref{eq.gamma1} also leads to
$$
|\tilde{f}'(z)|
\le \frac{|ad|+|bc|}{(|d|-|cz|)^2}
\le \frac{|\lambda|(1+\gamma_0)^2 + (\gamma_0|\lambda|)^2}{(1-2\gamma_0)^2}
%\le \frac{1+4\gamma_0}{1-8\gamma_0}|\lambda|
%\le \beta_0^{-1}|\lambda|
%\le |\Lambda| \le 1
$$
for all $|z| \le 1$. Hence, using \eqref{eq.gamma1} once more,
$$
|\tilde{f}'(z)|
\le \frac{1+4\gamma_0}{1-4\gamma_0}|\lambda|
\le \beta_0^{-1}|\lambda|
\le |\Lambda| \le 1
$$
That implies
\begin{equation}\label{eq.outrolado}
\tilde{f}\big(B\big(0,\frac{|b|}{5|d|}\big)\big) \subset B\big(\frac{b}{d},\frac{|b|}{5|d|}\big)
\end{equation}
From \eqref{eq.umlado} and \eqref{eq.outrolado} we get that
$g^{N_0}(B(0,r))\cap \tilde{f}(g^{N_0}(B(0,r)))=\emptyset$ for all $r\in[0,r_0]$.
This completes the proof of the lemma.
\end{proof}

\begin{lemma}\label{l.aproximacao2}
Given $\beta_0, \sigma_0 \in (0,1)$ there exist $\gamma_0>0$ and $N_0\in\N$ such that for
any $f(z)=\lambda z$ and $g(z)=\Lambda z$ with $|\lambda|\le \beta_0|\Lambda|$ and
$|\lambda|\le\sigma_0$, and for any $\tilde{g}$ whose inverse is a $\gamma_0$-deformation
of $g$, we have
$$
f^{N_0}(B(0,r))\cap \tilde{g}^{-1}(f^{N_0}(B(0,r)))=\emptyset
$$
for any $r\in[0,1]$ such that $\tilde{g}^{-1}(B(0,r))\nsubset f^{-1}(B(0,r))$.
\end{lemma}

\begin{proof}
Fix $N_0\ge 1$ such that
\begin{equation}\label{eq.N2}
|\lambda|^{N_0}
\le \sigma_0^{N_0}
\le \frac{1-\beta_0}{100}.
\end{equation}
Fix $\gamma_0>0$ such that
\begin{equation}\label{eq.gamma2}
\gamma_0=\frac{1-\beta_0}{100}
< \frac{1}{100}.
\end{equation}
Write $\tilde{g}(z)=(az+b)/(cz+d)$. %, that is , $\tilde{g}^{1}(z)=(dz-b)/(-cz+a)$.
Suppose that $b=0$. The assumption that $\tilde{g}^{-1}$ is a $\gamma_0$-deformation of
$g^{-1}$, together with \eqref{eq.gamma2}, gives
$$
|\tilde{g}^{-1}(z)|
\le \frac{|d z|}{|a|-|c|}
\le \frac{1+\gamma_0}{1-2\gamma_0}|\Lambda^{-1}z|
\le \beta_0 |\Lambda^{-1}z|
\le |\lambda^{-1}z|
$$
whenever $|z| \le 1$. This means that $\tilde{g}^{-1}(B(0,r)) \subset f^{-1}(B(0,r))$
for every $r \le 1$, in which case there is nothing to do.
Now, let us suppose that $b\neq 0$. Take
$$
r_0 = \frac{10|b||\lambda|}{(1-\beta_0)}.
$$
Then $|\tilde{g}^{-1}(z)| \le |\lambda^{-1} z|$ for every $|z| \in [r_0,1]$. Indeed,
$$
|\tilde{g}^{1}(z)|
\le \frac{|dz|+|b|}{|a|-|c|}
\le \frac{|\Lambda^{-1}|(1+\gamma_0)+|(1-|\beta_0|)/(10|\lambda|)}{1-2\gamma_0}|z|
$$
and, in view of \eqref{eq.gamma2}, the right hand side is bounded by
$$
\frac{\beta_0 (1 + \gamma_0) + 10\gamma_0}{1-2\gamma_0}|\lambda^{-1} z|
\le \frac{\beta_0 + 20\gamma_0}{1-2\gamma_0}|\lambda^{-1} z|
\le |\lambda^{-1} z|.
$$
This means that $\tilde{g}^{-1}(B(0,r)) \subset f^{-1}(B(0,r))$ for every $r \in [r_0, 1]$.
Now consider $r\in[0,r_0]$. By \eqref{eq.N2}
$$
|\lambda|^{N_0} r \le |\lambda|^{N_0} r_0
\le \frac{(1-\beta_0)}{100} \frac{10|b||\lambda|}{(1-\beta_0)}
\le \frac{|b\lambda|}{10} \le \frac{|b\lambda|}{5|a|}
$$
and that means that
\begin{equation}\label{eq.umlado2}
f^{N_0}(B(0,r))
\subset B\big(0,\frac{|b\lambda|}{5|a|}\big)
\subset B\big(0,\frac{|b|}{5|a|}\big).
\end{equation}
Recalling that $|\lambda|\le\min\{1,|\Lambda|\}$, the relation \eqref{eq.gamma2} also gives
$$
\begin{aligned}
|(\tilde{g}^{-1})'(z)|
& \leq \frac{|ad|+|bc|}{(|a|-|c|)^2}
\le \frac{|\Lambda^{-1}|(1+\gamma_0)^2 + \gamma_0^2}{(1-2\gamma_0)^2}
& % \le \frac{\beta_0^{-1}(1+4\gamma_0)}{1-4\gamma_0}|\lambda|^{-1}
\le 2|\lambda^{-1}|
\end{aligned}
$$
for all $|z| \le 1$. This implies
\begin{equation}\label{eq.outrolado2}
\tilde{g}^{-1}\big(B\big(0,\frac{|b\lambda|}{5|a|}\big)\big)
\subset B\big(\frac{b}{a},\frac{2|b|}{5|a|}\big).
\end{equation}
From \eqref{eq.umlado2} and \eqref{eq.outrolado2} we get that
$f^{N_0}(B(0,r))\cap \tilde{g}^{-1}(f^{N_0}(B(0,r)))=\emptyset$ for every $r\in[0, r_0]$.
This completes the proof of the lemma.
\end{proof}

\subsection{Proof of Proposition~\ref{p.oitenta}}

If $d(\sA,\sB)<\gamma$ then every $\hat B_x^{-1}$ is a $(C\gamma)$-deformation
of $f=\hat A_x^{-1}$, where the constant $C=\sup_{x\in\cX}|\theta_x|$ depends only on $\sA$.
Indeed,
$$
B_x=\left(\begin{array}{cc} a_x & b_x \\ c_x & d_x\end{array}\right)
\quad\text{with}\quad
|a_x - \theta_x|, |b_x|, |c_x|, |d_x - \theta_x^{-1}| < \gamma
$$
yields
$$
\hat{B}_x^{-1}
 = \frac{d_x z - b_x}{-c_x z + a_x}
 = \frac{d_x\theta_x^{-1} z - b_x\theta_x^{-1}} {-c_x\theta_x^{-1} z + a_x\theta_x^{-1}}
$$
with $|d_x\theta_x^{-1} - \theta_x^{-2}|$, $|b_x\theta_x^{-1}|$, $|c_x\theta_x^{-1}|$,
$|a_x\theta_x^{-1} - 1| < \gamma |\theta_x|^{-1} \le C \gamma |\theta_x|^{-2}$.
Take
$$
f=\hat{A}_x^{-1} \quand g=\hat{D}_x^{-1} \quad\text{for each $x\in\cX_-$.}
$$
Observe that $f(z)=|\theta_x|^{-2}|z|$ and $g(z)=\sigma_x^{-2}|z|$.
Since $\sigma_x \le \beta|\theta_x|$ and $\sigma_x \ge \sigma^{-1}$ (see Lemma~\ref{l.lemmakey1}),
we may apply Lemma~\ref{l.aproximacao} with $\beta_0=\beta^2$ and $\sigma_0=\sigma^2$.
Using also the observation  in the previous paragraph, we get that there exist $\gamma_->0$
and $N_-\in\N$ such that
$$
\hat{B}_x^{-1}(\hat{D}_x^{-N_-}(B(0,r)))\cap\hat{D}_x^{-N_-}(B(0,r))=\emptyset
\quad\text{for $x\in\cX_-$}
$$
if $d(\sA,\sB)<\gamma_-$ and $r\in[0,1]$ is such that
$\hat{B}_x^{-1}(B(0,r))\nsubset\hat{D}_x^{-1}(B(0,r))$. Now take
$$
f=\hat{D}_x \quand g=\hat{A}_x \quad\text{for each $x\in\cX_+$.}
$$
Then $f(z)=\sigma_x^{2}|z|$ and $g(z)=|\theta_x|^{2}|z|$ and so we are in the setting of
Lemma~\ref{l.aproximacao2}, with $\beta_0=\beta^2$ and $\sigma_0=\sigma^2$.
In this way we find $\gamma_+>0$ and $N_+\in\N$ such that
$$
\hat{B}_x^{-1}(\hat{D}_x^{N_+}(B(0,r)))\cap\hat{D}_x^{N_+}(B(0,r))=\emptyset
\quad\text{for $x\in\cX_+$}
$$
if $d(\sA,\sB)<\gamma_+$ and $r\in[0,1]$ is such that
$\hat{B}_x^{-1}(B(0,r))\nsubset\hat{D}_x^{-1}(B(0,r))$.
To complete the proof of the proposition,
just take $\gamma=\min\{\gamma_-,\gamma_+\}$ and $N=\max\{N_-,N_+\}$.

\section{Consequences of Theorem~\ref{t.Bernoulli}}\label{s.consequences}

In this section we deduce Theorem~\ref{t.generalcase} and Theorem~\ref{t.convergenciaemmedida}.

\subsection{Proof of Theorem~\ref{t.generalcase}}\label{ss.generalcase}

The main step is the following lemma. Let $\lambda$ be the Lebesgue measure on the
unit interval $I$, and let $\|\eta\|$ denote the total variation of a  signed  measure $\eta$.

\begin{lemma}[Avila]\label{l.pullback}
Let $X$ be a metric space such that every bounded closed subset is compact, and let $\nu$
be any Borel probability measure in $X$ whose support $Z=\supp\nu$ is bounded.

For every $\vep>0$ there is $\delta>0$ and a weak$^*$ neighborhood $V$ of $\nu$
such that every probability $\mu\in V$ whose support is contained in $B_\delta(Z)$ may be
written as $\phi_* q= \mu$ for some probability $q$ on $Z\times I$ satisfying
$\|q - (\nu\times\lambda)\| < \vep$ and some measurable map $\phi:Z\times I \to X$ such
that $d(\phi(x,t),x) < \vep$ for all $x\in Z$.
\end{lemma}

\begin{proof}
We claim that for any $\delta>0$ there exists a cover $\cQ$ of $B_\delta(Z)$ by disjoint
measurable sets $Q_i$,  $i=1, \ldots, n$ with $\nu(Q_i)>0$ and $\nu(\partial Q_i)=0$ and
$\diam Q_i < 12\delta$. This can be seen as follows.
For each $x\in Z$ take $r_x\in(\delta,2\delta)$ such that $\nu(\partial B(x,r_x))=0$.
Then $\{B(x,r_x):\,x\in Z \}$ is a cover of the closure of $B_\delta(Z)$, a bounded closed set.
Let $\{V_1,V_2,\dots,V_k\}$ be a finite subcover. By construction, $\diam V_i< 4\delta$ and
$\nu(V_i)>0$ and $\nu(\partial V_i)=0$ for every $i$.
Consider the partition $\cP$ of $\cup_{i=1}^{k}V_i$ into the sets $V_1^* \cap \cdots \cap V_k^*$,
where each $V_i^*$ is either $V_i$ or its complement. Define
$$
Q_1=V_1\cup\{P\in\cP:\,\nu(P)=0 \text{ and } P\subset V_i \text{ with } V_i\cap V_1\not =\emptyset\}.
$$
Then define $Q_2\subset X$ as follows. If $V_2\subset Q_1$ then $Q_2=\emptyset$;
otherwise, notice that $\nu(V_2\setminus Q_1)>0$, and then take
$$
Q_2=V_2\cup\{P\in\cP:\,\nu(P)=0 \text{ and } P\subset V_i \text{ with } V_i\cap V_2\not =\emptyset\}\setminus Q_1
$$
More generally, for every $2 \le l \le k$, assume that $Q_1$, \dots, $Q_{l-1}$ have been defined
and then let $Q_l=\emptyset$ if $V_l\subset \cup_{i=1}^{l-1}Q_i$ and
$$
Q_l=V_l\cup\{P\in\cP: \nu(P)=0 \text{ and } P\subset V_i \text{ with } V_i\cap V_l\not =\emptyset\}\setminus \cup_{i=1}^{l-1}Q_i
$$
if $\nu(V_l\setminus \cup_{i=1}^{l-1}Q_i)>0$. Those of these sets $Q_i$ that are non-empty form
a cover $\cQ$ as in our claim.

Proceeding with the proof of the lemma, take $\delta=\vep/12$ and assume that the neighborhood
$V$ is small enough that
$$
\sum_{i=1}^n |\mu(Q_i) - \nu(Q_i)| < \vep \quad\text{for every } \mu \in V.
$$
Let $Z_i=\supp\nu \cap Q_i$ for each $i=1, \ldots, n$.
Clearly, $\nu(Z_i)=\nu(Q_i)$. Let $q$ be the measure on $Z \times I$ that coincides with
$$
\frac{\mu(Q_i)}{\nu(Q_i)} (\nu \times \lambda)
$$
restricted to each $Z_i \times I$. For each $i$, let $a_{i,j}$, $j\in J(i)$ be the atoms of $\mu$
contained in $Q_i$ (the set $J(i)$ may be empty). Moreover, let
$I_{i,j}$, $j\in J(i)$ be disjoint subsets of $I$ such that
$$
\lambda(I_{i,j}) = \frac{p_{i,j}}{\mu(Q_i)} \quad\text{for all } j\in J(i),
$$
where $p_{i,j}=\nu(a_{i,j})$. Denote $I_{i}=I\setminus \cup_{j\in J(i)} I_{i,j}$. Then
$$
q\big(Z_i \times I_{i} \big)
%= \frac{\mu(Q_i)}{\nu(Q_i)} \nu(Q_i) \lambda(I_{i})
%= \mu(Q_i) \left(1 - \sum_{j\in J(i)} \frac{p_{i,j}}{\mu(Q_i)}\right)
= \mu(Q_i) - \sum_{j\in J(i)}  p_{i,j}
= \mu\big(Q_i\setminus \{a_{i,j}: j\in J(i)\}).
$$
The assumption implies that $X$ is a polish space, that is, a complete separable metric
space. Since all Borel non-atomic probabilities on polish spaces are isomorphic
(see  Ito~\cite[\S~2.4]{Ito84}), the previous equality ensures that there exists an invertible
measurable map
$$
\phi_i: Z_i \times I_{i} \to Q_i \setminus \{a_{i,j}: j\in J(i)\}
$$
mapping the restriction of $q$ to the restriction of $\mu$. By setting $\phi \equiv a_{i,j}$
on each $Z_i \times I_{i,j}$ we extend $\phi_i$ to a measurable map $Z_i \times I \to Q_i$
that  still sends  the restriction of $q$ to the restriction of $\mu$.
Gluing all these extensions we obtain a measurable map  $\phi: Z \times I \to X$ such that
$\phi_*q = \mu$.
By construction, $\phi(x,t) \in Q_i$ for every $x\in Z_i$ and $t\in I$.
This implies that $d(\phi(x,t), x) \le \diam Q_i < \vep$ for all $(x,t)\in Z\times I$. Finally,
$$
\begin{aligned}
\|q - (\nu \times \lambda)\|
% = \sum_{i=1}^n \big\| \frac{\mu(Q_i)}{\nu(Q_i)} (\nu \times \lambda)\mid(Z_i\times I)
 %                                                                              - (\nu \times \lambda)\mid(Z_i\times I)\big\|
& = \sum_{i=1}^n \big\| \big(\frac{\mu(Q_i)}{\nu(Q_i)} - 1\big) (\nu \times \lambda)\mid(Z_i\times I)\big\| \\
& = \sum_{i=1}^n | \mu(Q_i) - \nu(Q_i) |  < \vep.
\end{aligned}
$$
The proof of the lemma is complete.
\end{proof}

Now, given $\rho>0$, let $\nu$ be a probability measure in $\GL(2,\C)$ with compact support.
Consider $\cX=\supp\nu \times I$, $\sps=\nu\times \lambda$ and $\sA:\cX\to \GL(2,\C)$ given by $\sA(x,t)=x$.
From Theorem~\ref{t.Bernoulli}, there is $\vep>0$ such that $|\lambda_\pm(\sA,\sps)-\lambda_\pm(\sB,\sqs)|<\rho$
for all $(\sB,\sqs)$ such that $d(\sps,\sqs)<\vep$ and $d(\sA,\sB)<\vep$.
On the other hand, Lemma~\ref{l.pullback} implies that there exist a weak$^*$ neighborhood $V$ and $\delta$ such
that if $\nu'\in V$ and $\supp\nu'\subset B_\delta(\supp\nu)$ then there exist $\sB:\cX\to \GL(2,\C)$ and a probability measure $\sqs$ on $\cX$ such that $d(\sps,\sqs)<\vep$, $d(\sA,\sB)<\vep$ and $\nu'=\sB_*\sqs$.
Noting that $\lambda_\pm(\nu)=\lambda_\pm(\sA,\sps)$ and $\lambda_\pm(\nu')=\lambda_\pm(\sB,\sqs)$,
we obtain Theorem~\ref{t.generalcase}.

\subsection{Proof of Theorem~\ref{t.convergenciaemmedida}}

Recall that we denote $M=\cX^\Z$.

\begin{lemma}\label{l.compactprohorov}
Let $(\mu^k)_k$ be a sequence of probabilities on $M$ converging to $\mu$ in the weak$^*$ topology.
Let $(m^k)_k$ be a sequence of probabilities on $M\times \Proj(\C^2)$ projecting down to $(\mu^k)_k$.
Then there exists a subsequence of $(m^k)_k$ converging, in the weak$^*$ topology,
to some probability $m$ that projects down to $\mu$.

In particular, the space $\cM(p)$ of probabilities measures on $M \times \Proj(\C^2)$ that project down
to $\mu$ is compact for the weak$^*$ topology.
\end{lemma}

\begin{proof}
Since $M$ and $M\times \Proj(\C^2)$ are polish spaces, we may apply Prohorov's theorem
(see Billingsley~\cite[\S~6]{Bil68}) in either of these spaces: a sequence of probabilities $(\xi^k)_k$
has weak$^*$-converging subsequences if and only if for each $\vep>0$ there is a compact set
$K_\vep$ such that $\xi^k(K_\vep)>1-\vep$ for any $k \ge 1$.
By assumption, $(\mu^k)_k$ converges to $\mu$ in the weak$^*$ topology. Thus, given any $\vep>0$,
there is some compact set $K_\vep\subset M$ such that $\mu^k(K_\vep)>1-\vep$,
for any $k\ge1$. Then $\hat{K}_\vep=K_\vep\times \Proj(\C^2)$ is compact and
$m^k(\hat{K}_\vep)=\mu^k(K_\vep)>1-\vep$ for any $k\ge1$. This ensures that $(m^k)_k$ has
weak$^*$-converging subsequences, as claimed. Considering the special case when the sequence
$(\mu^k)_k$ is constant equal to $\mu$, one gets the last part of the lemma.
\end{proof}
%\begin{remark}\label{r.convergenciafraca}
%Observe that if $(\sps_k)_k$ converges weakly$^*$ to $\sps$, then $\sps_k^\Z$ converges weakly$^*$
% to $\sps^\Z$, where $\sps_k^\Z$ and $\sps^\Z$ stand for the product measure on $\cX^\Z$.
%\end{remark}

\begin{lemma}\label{l.imagensconvergem}
Let $\sA^k:\cX\to\GL(2,\CC)$, $k\ge 1$ be such that $d(\sA^k,\sA)\to 0$, and let
$F_{A^k}:M\times\Proj(\CC^2)\to M\times\Proj(\CC^2)$, $k\ge 1$ be the associated projective
cocycles. Let $(m^k)_k$ be a sequence of probability measures on $M\times\Proj(\CC^2)$
such that $m^k$ projects down to $\mu$ for all $k$.
If $(m^k)_k$ converges to $m$, then $((F_{\sA^k})_*m^k)_k$ converges to $(F_\sA)_*m$,
in the weak$^*$ topology.
\end{lemma}

\begin{proof}
Let $\varphi:M\times \Proj(\C^2)\to \R$ be any uniformly continuous bounded function.
By the theorem of Lusin, given any $\vep>0$ there is some compact set $K\subset M$ such that
$\mu(K)>1-\vep$ and $\sA:M \to \SL(2,\C)$ is continuous restricted to $K$.
Then $\varphi\circ F_\sA:K\times \Proj(\C^2)\to \R$ is continuous and so, by the extension
theorem of Tietze, it admits some continuous extension $\tilde{\varphi}:M\times \Proj(\C^2)\to \R$
to the whole space, with $\|\tilde{\varphi}\| \le \|\varphi\|$. We have
$$\begin{aligned}
|\int \varphi\,d(F_{\sA^k})_*m^k -\int \varphi\,d(F_{\sA})_*m|
=|\int \varphi\circ F_{\sA^k}\,dm^k -\int \varphi\circ F_\sA\,dm|
&\\\le |\int \varphi\circ F_{\sA^k}\,dm^k -\int \varphi\circ F_\sA\,dm^k|
+|\int \varphi\circ F_{\sA}\,dm^k -\int \varphi\circ F_\sA\,dm|
\end{aligned}
$$
The first term on the right hand side converges to zero when $k\to\infty$,
because $\varphi\circ F_{\sA^k}$ converges uniformly to $\varphi\circ F_{\sA}$.
The last term admits the following bound:
$$
\begin{aligned}
 |\int \varphi\circ F_{\sA}&\,dm^k -\int \varphi\circ F_\sA\,dm| \\
% & \le|\int \tilde{\varphi}\,dm^k -\int \tilde{\varphi}\,dm|+
% |\int (\varphi\circ F_{\sA}-\tilde{\varphi})\,dm^k -\int (\varphi\circ F_\sA-\tilde{\varphi})\,dm| \\
 & \le |\int \tilde{\varphi}\,dm^k -\int \tilde{\varphi}\,dm|+2\|\varphi\|(m^k+m)(K^c\times \Proj(\C^2)) \\
 & \le |\int \tilde{\varphi}\,dm^k -\int \tilde{\varphi}\,dm|+4\|\varphi\|\mu(K^c)
 % \\ & \le |\int \tilde{\varphi}\,dm^k -\int \tilde{\varphi}\,dm|+4\|\varphi\|\vep
\end{aligned}
$$
The first term on the right hand side converges to zero when $k\to\infty$,
because $\tilde\varphi$ is continuous, and the second term is bounded by $4\|\varphi\|\vep$.
Since $\vep>0$ is arbitrary, this proves that
$$
\int \varphi\,d(F_{\sA^k})_*m^k \to \int \varphi\,d(F_{\sA})_*m
\quad\text{as $k\to\infty$,}
$$
for any uniformly continuous bounded function $\varphi$.
So (see Theorem~2.1 in Billingsley~\cite{Bil68}), the sequence
$(F_{\sA^k})_*m^k$ converges weakly$^*$ to $(F_{\sA})_*m$, as claimed.
\end{proof}

\begin{corollary}\label{c.uestadosconvergem}
Suppose that $\lambda_+(\sA,\sps)>0$ and let $m^u$ be the $u$-state defined by
\eqref{eq.mus}. Let $(A^k)_k$ be such that $d(A^k,A)\to 0$ as $k\to\infty$.
For each $k\ge 1$, let $m^u_k$ be an invariant  $u$-state for $(\sA^k,\sps)$ realizing
$\lambda_+(A^k,p)$. Then $(m^u_k)_k$ converges to $m^u$ in the weak$^*$ topology.
\end{corollary}

\begin{proof}
In view of the compactness Lemma~\ref{l.compactprohorov}, we only have to show that every
accumulation point $m$ of the sequence $(m^u_k)_k$ coincides with $m^u$. Restricting to a
subsequence if necessary, we may suppose that $m$ is the limit of $(m^u_k)_k$, not just an
accumulation point. We claim that $m$ is an $F$-invariant probability.
By definition, every $m_k^u$ projects down to $\mu$. Then, we may use
Lemma~\ref{l.imagensconvergem} to conclude that $(F_{\sA^k})_*m^u_k$ converges to
$(F_\sA)_*m$ as $k\to\infty$. Since each $m_k^u$ is assumed to be $F_{\sA^k}$-invariant,
this proves that $(F_A)_*m=m$, as claimed. The assumption implies that $(\phi_{A^k})_k$
converges to $\phi_A$, uniformly on $M\times\Proj(\CC^2)$. Consequently,
$$
\int\phi_A\,dm=\lim\int \phi_{\sA^k}\,dm^u_k.
$$
In addition, using Theorem~\ref{t.Bernoulli}:
$$
\lim \int \phi_{\sA^k}\,dm^u_k
=\lim \lambda_+(\sA^k,\sps)
=\lambda_+(\sA,\sps)
=\int \phi_\sA\,dm^u
$$
This proves that $m$ realizes $\lambda_+(A,p)$. From Remark~\ref{r.uniqueustaterealizing} we
conclude that $m=m^u$. This completes the proof of the corollary.
\end{proof}

Let us deduce Theorem~\ref{t.convergenciaemmedida}. We only have to show that
$$
\mu(\{x\in M: \angle(E_{\sA,x}^u,E_{\sA^k,x}^{u})<\vep\}) \to 1 \quad\text{when } k\to\infty,
$$
as the statement about stable spaces $E^s_A$ is analogous. By Theorem~\ref{t.Bernoulli},
the assumption $\lambda_+(A,p)>0$ implies $\lambda_+(A^k,p)>0$ for every large $k$.
Let $m^u$ and $m^u_k$, $k\ge 1$ be $u$-states for $A$ and $A^k$, $k\ge 1$ defined
as in \eqref{eq.mus}. By Corollary~\ref{c.uestadosconvergem}, $(m_k^u)_k$ converges to
$m^{u}$ in the weak$^*$ topology. The map $\psi:M\to\Proj(\CC^2)$, $\psi(x) = E_{\sA,\bx}^{u}$
is measurable map and its graph has full $m^u$-measure.
By the theorem of Lusin, given any $\vep>0$ we may find a compact set $K\subset M$
such that the restriction $\psi_K$ to $K$ is continuous and the $m^u$-measure of its graph
is bigger than $1-\vep$. Given $\delta>0$, let $V$ be an open neighborhood of the graph
of $\psi_K$ %whose intersection with every fiber has diameter less than $\delta$, that is,
such that
$$
V\cap \big(K\times \Proj(\C^2)\big)\subset V_\delta:=\{(\bx,\xi)\in K\times \Proj(\C^2):\angle(\xi,\psi(\bx))<\delta\}.
$$
By the definition of $m_k^u$,
$$
m_k^u(V_\delta)=\mu(\{\bx \in K: d(E^{u}_{\sA_{k},\bx},E^{u}_{\sA,\bx})<\delta\}).
$$
By weak$^*$ convergence, $\liminf m_k^u(V) \ge m^u(V) \ge 1-\vep$. On the other hand,
$m^u_k(K\times \Proj(\C^2))=\mu(K) \ge 1-\vep$ for every $k$. Thus,
$$
m_k^u\big(V_\delta\big)
\ge m_k^u\big(V\cap \big(K\times \Proj(\C^2)\big)\big)
\ge 1-3\vep
$$
for every large $k$. Hence,
$\mu(\{\bx \in M: d(E^{u}_{\sA_{k},\bx},E^{u}_{\sA,\bx})<\delta\})\ge 1-3\vep$ for every large $k$.
Since $\delta$ and $\vep$ are arbitrary, this proves Theorem~\ref{t.convergenciaemmedida}.

\section{Concluding remarks}\label{s.final}

We are going to describe a construction of points of discontinuity of
the Lyapunov exponents as functions of the cocycle, relative to some
H\"older topology. This builds on and refines~\cite{Boc-un,Boc02,BcV05,Man83},
where it is shown that Lyapunov exponents are often discontinuous relative
to the $C^0$ topology. In the final section we list a few related open problems.

\subsection{An example of discontinuity}\label{ss.discontinuity}

Let $M=\Sigma_2$ be the shift with $2$ symbols,
endowed with the metric $d(\bx,\by) = 2^{-N(\bx,\by)}$, where
$$
N(\bx,\by) = \sup\{n\ge 0: \sx_n=\sy_n \text{ whenever } |n| < N\}.
$$
For any $r\in(0,\infty)$, the $C^r$ norm in the space of $r$-H\"older
continuous functions $L:M \to \cL(\CC^d,\CC^d)$ is defined by
$$
\|L\|_r = \sup_{\bx \in M} \|L(\bx)\| + \sup_{\bx \neq \by} \frac{\|L(\bx) - L(\by)\|}{d(\bx, \by)^r}.
$$
Consider on $M$ the Bernoulli measure associated to any probability
vector $(p_1, p_2)$ with positive entries and $p_1 \neq p_2$.
Given any $\sigma>1$, consider the (locally constant) cocycle $A: M \to \SL(2,\RR)$
defined by
$$
A(\bx) = \left(\begin{array}{cc}\sigma & 0 \\ 0 & \sigma^{-1}\end{array}\right)
\quad\text{if } \sx_0=1
$$
and
$$
A(\bx) = \left(\begin{array}{cc}\sigma^{-1} & 0 \\ 0 & \sigma\end{array}\right)
\quad\text{if } \sx_0=2.
$$

\begin{theorem}\label{t.discontinuity}
For any $r>0$ such that $2^{2r} < \sigma$ there exist $B:M\to \SL(2,\RR)$
with vanishing Lyapunov exponents and such that $\|A-B\|_r$ is arbitrarily close
to zero.
\end{theorem}

Since the Lyapunov exponents $\lambda_\pm(A) = \pm |p_1 - p_2| \log \sigma $ of
$A:M\to\SL(2,\RR)$ are non-zero, it follows that $A$ is a point of discontinuity for
the Lyapunov exponents relative to the $C^r$ topology.

The proof of Theorem~\ref{t.discontinuity} is an adaptation of ideas of
Knill~\cite{Kni91} and Bochi~\cite{Boc-un,Boc02}. Here is an outline.
Notice that the unperturbed cocycle $A$ preserves both the horizontal line
bundle $H_\bx=\{\bx\}\times \RR(1,0)$ and the vertical line bundle
$V_\bx=\{\bx\}\times\RR(0,1)$. Then, the Oseledets subspaces must coincide
with $H_\bx$ and $V_\bx$ almost everywhere. We choose cylinders $Z_n\subset M$
whose first $n$ iterates $f^i(Z_n)$, $0 \le i \le n-1$ are pairwise disjoint.
Then we construct cocycles $B_n$ by modifying $A$ on some of these iterates
so that
\begin{equation*}%\label{eq.troca}
B_n^n(x)H_\bx = V_{f^n(\bx)} \quad\text{and}\quad
B_n^n(x)V_\bx = H_{f^n(\bx)} \quad\text{for all $\bx\in Z_n$.}
\end{equation*}
We deduce that the Lyapunov exponents of $B_n$ vanish.
Moreover, by construction, each $B_n$ is constant on every atom of some finite
partition of $M$ into cylinders. In particular, $B_n$ is H\"older continuous
for every $r>0$. From the construction we also get that
\begin{equation}\label{eq.rnorm}
\|B_n - A\|_r \leq \const \left(2^{2r}/\sigma\right)^{n/2}
\end{equation}
decays to zero as $n\to\infty$. This is how we get the claims in the theorem.
Now let us fill-in the details of the proof.

Let $n=2k+1$ for some $k\ge 1$ and $Z_n=[0;2, \dots, 2, 1, \dots, 1, 1]$ where
the symbol $2$ appears $k$ times and the symbol $1$ appears $k+1$ times.
Notice that the $f^i(Z_n)$, $0 \le i \le 2k$ are pairwise disjoint. Let
\begin{equation}\label{eq.perturbation}
\vep_n = \sigma^{-k} \quand \delta_n = \arctan \vep_n.
\end{equation}
Define $R:M \to \SL(2,\RR)$ by
$$\begin{aligned}
R(\bx) & = \text{rotation of angle } \delta_n
\quad\text{if } \bx \in f^{k}(Z_n) \\
R(\bx) & = \left(\begin{array}{cc} 1 & 0 \\ \vep_n & 1\end{array}\right)
\quad\text{if } \bx \in Z_n \cup f^{2k}(Z_n) \\
R(\bx) & = \id \quad\text{in all other cases.}
\end{aligned}$$
and then take $B_n=AR_n$.

\begin{lemma}\label{l.troca}
$B_n^{n}(\bx)H_\bx = V_{f^{n}(\bx)}$ and $B_n^{n}(\bx)V_\bx = H_{f^{n}(\bx)}$
for all $\bx\in Z_n$.
\end{lemma}

\begin{proof}
Notice that for any $\bx\in Z_n$,
$$\begin{aligned}
B_n^k(\bx)H_\bx & = \RR (\vep_n, 1) \quand B_n^k(\bx)V_\bx = V_{f^k(\bx)} \\
B_n^{k+1}(\bx)H_\bx & = V_{f^{k+1}(\bx)} \quand B_n^{k+1}(\bx)V_\bx = \RR (-\vep_n,1)\\
B_n^{2k}(\bx)H_\bx & = V_{f^{2k}(\bx)} \quand B_n^{2k}(\bx)V_\bx = \RR (-1, \vep_n).
\end{aligned}$$
The claim follows by iterating one more time.
\end{proof}

\begin{lemma}\label{l.norma}
There exists $C>0$ such that $\|B_n - A\|_r \le C \left(2^{2r}/\sigma\right)^k$
for every $n$.
\end{lemma}

\begin{proof}
Let $L_n=A-B_n$. Clearly, $\sup\|L\| \le \sup\|A\|\,\|\id - R_n\|$
and this is bounded by $\sigma \vep_n$. Now let us estimate the second term
in the definition \eqref{eq.rnorm}. If $\bx$ and $\by$ are not in the same  cylinder
$[0;a]$ then $d(\bx,\by)=1$, and so
\begin{equation}\label{eq.bound1}
\frac{\|L_n(\bx) - L_n(\by)\|}{d(\bx, \by)^r}
\le 2 \sup \|L_n\| \le 2\sigma \vep_n.
\end{equation}
From now on we suppose $\bx$ and $\by$ belong to the same cylinder.
Then, since $A$ is constant on cylinders,
$$
\frac{\|L_n(\bx) - L_n(\by)\|}{d(\bx, \by)^r}
 = \frac{\|A(\bx)(R_n(\bx)-R_n(\by))\|}{d(\bx, \by)^r}
 \le \sigma \frac{\|R_n(\bx)-R_n(\by)\|}{d(\bx, \by)^r}.
$$
If neither $\bx$ nor $\by$ belong to $Z_n \cup f^k(Z_n) \cup f^{2k}(Z_n)$
then $R_n(\bx)$ and $R_n(\by)$ are both equal to $\id$, and so the
expression on the right vanishes.
If $\bx$ and $\by$ belong to the same $f^i(Z_n)$ then $R_n(\bx)=R_n(\by)$
and so, once more, the expression on the right vanishes.
We are left to consider the case when one of the points belongs to some
$f^i(Z_n)$ and the other one does not. Then $d(\bx,\by)\ge 2^{-2k}$ and
so, using once more that $\|\id - R_n\| \le \vep_n$ at every point,
$$
\frac{\|L_n(\bx) - L_n(\by)\|}{d(\bx, \by)^r}
\le \sigma \frac{\|R_n(\bx)-R_n(\by)\|}{d(\bx, \by)^r}
\le 2 \sigma \vep_n 2^{2kr}.
$$
Noting that this bound is worst than \eqref{eq.bound1}, we conclude that
$$
\|L_n\|_ r
 \le \sigma\vep_n + 2 \sigma \vep_n 2^{2kr}
 \le % 3 \sigma \vep_n 2^{2kr} =
 3 \sigma \left(2^{2r}/\sigma\right)^k
$$
Now it suffices to take $C=3\sigma$.
\end{proof}

Now we want to prove that $\lambda_\pm(B_n)=0$ for every $n$.
Let $\mu_n$ be the normalized restriction of $\mu$ to $Z_n$ and
$f_n:Z_n \to Z_n$ be the first return map (defined on a full
measure subset). Indeed,
$$
Z_n=\bigsqcup_{b\in \cB} [0;w,b,w]
\quad \text{(up to a zero measure subset)}
$$
where $w=(1, \dots, 1, 2, \dots, 2, 2)$ and the union is over the
set $\cB$ of all finite words $b=(b_1, \dots, b_s)$ not having
$w$ as a sub-word. Moreover,
$$
f_n \mid [0;w, b, w] = f^{n+s} \mid [0;w, b, w]
\quad\text{for each $b\in\cB$.}
$$
Thus, $(f_n,\mu_n)$ is a Bernoulli shift with an infinite alphabet
$\cB$ and probability vector given by $p_b=\mu_n([0;w, b, w])$.
Let $\hat B_n: Z_n \to \SL(2,\RR)$ be the cocycle induced by $B$
over $f_n$, that is,
$$
\hat B_n \mid [0;w, b, w] = B^{n+s}_n \mid [0;w, b, w]
\quad\text{for each $b\in\cB$.}
$$
It is a well known basic fact (see~\cite[Proposition~2.9]{Vi07}, for instance)
that the Lyapunov spectrum of the induced cocycle is obtained multiplying
the Lyapunov spectrum of the original cocycle by the average return time.
In our setting this means
$$
\lambda_\pm(\hat B_n) = \frac{1}{\mu(Z_n)} \lambda_\pm(B_n).
$$
Therefore, it suffices to prove that $\lambda_\pm(\hat B_n)=0$ for every $n$.

Indeed, suppose the Lyapunov exponents of $\hat B_n$ are non-zero and let
$E^u_\bx \oplus E^s_\bx$ be the Oseledets splitting (defined almost everywhere
in $Z_n$). Consider the probability measures $m^u$ and $m^s$ for the cocycle
$\hat B_n$ defined as in \eqref{eq.mus}. The key observation is that, as a
consequence of Lemma~\ref{l.troca}, the cocycle $\hat B_n$ permutes the vertical
and horizontal subbundles:
\begin{equation}\label{eq.troca}
\hat B_n(\bx)H_\bx = V_{f_n(\bx)} \quand \hat B_n(\bx)V_\bx = H_{f_n(\bx)}
\quad\text{for all $\bx\in Z_n$.}
\end{equation}
Let $m$ be the measure defined on $M\times\Proj(\RR^2)$ by
$$
m_n(X) = \frac 12 (\mu_n\left(\{\bx\in Z_n: V_\bx\in X \}\right)+\mu_n\left(\{\bx\in Z_n: H_\bx \in X \}\right).
$$
That is, $m_n$ projects down to $\mu_n$ and its disintegration is given
by $\bx \mapsto (\delta_{H_\bx} + \delta_{V_\bx})/2$. It is clear from \eqref{eq.troca}
that $m_n$ is $\hat B_n$-invariant.

\begin{lemma}\label{l.ergodicity}
The probability measure $m_n$ is ergodic.
\end{lemma}

\begin{proof}
Suppose there is an invariant set $\cX\subset M \times\Proj(\RR^2)$ with $m_n(\cX)\in(0,1)$.
Let $X_0$ be the set of $\bx\in Z_n$ whose fiber $\cX \cap (\{\bx\}\times\Proj(\RR^2))$
contains neither $H_\bx$ nor $V_\bx$. In view of \eqref{eq.troca}, $X_0$ is an
$f_n$-invariant set and so its $\mu_n$-measure is either $0$ or $1$.
Since $m_n(\cX)>0$, we must have $\mu_n(X_0)=0$. The same kind of argument shows
that $\mu_n(X_2)=0$, where $X_2$ is the set of $\bx\in Z_n$ whose fiber contains both
$H_\bx$ and $V_\bx$. Now let $X_H$ be the set of $\bx\in Z_n$ whose fiber contains
$H_\bx$ but not $V_\bx$, and let $X_V$ be the set of $\bx\in Z_n$ whose fiber contains
$V_\bx$ but not $H_\bx$. The previous observations show that $X_H \cup X_V$ has full
$\mu_n$-measure and it follows from \eqref{eq.troca} that
$$
f_n(X_H) = X_V \quand f_n(X_V)=X_H.
$$
Thus, $\mu_n(X_H) = 1/2 = \mu_n(X_V)$ and $f_n^2(X_H)=X_H$ and $f_n^2(X_V)=X_V$.
This is a contradiction because $f_n$ is Bernoulli and, in particular,
the second iterate is ergodic.
\end{proof}

By Lemma~\ref{l.combinacaolinear}, the invariant measure $m_n$ is a linear combination
of $m^u$ and $m^s$. Then, in view of Lemma~\ref{l.ergodicity}, $m_n$ must coincide with
either $m^s$ and $m^u$. This is a contradiction, because the conditional probabilities
of $m_n$ are supported on exactly two points on each fiber, whereas the conditional
probabilities of either $m^u$ and $m^s$ are Dirac masses on a single point.
This contradiction proves that the Lyapunov exponents of $\hat B_n$ do vanish for
every $n$, and that concludes the proof of Theorem~\ref{t.discontinuity}.

The same kind of argument shows that, in general, one can expect continuity to hold when
some of the probabilities $p_i$ vanishes:

\begin{remark}\label{r.vanishingpi}(Kifer~\cite{Ki82b})
Take $d=2$, a probability vector $\sps=(p_1,p_2)$ with non-negative coefficients,
and a cocycle $\sA=(A_1,A_2)$ defined by
$$
A_1 = \left(
         \begin{array}{cc}
           \sigma & 0 \\
           0 & \sigma^{-1} \\
         \end{array}
       \right) \quad \text{and} \quad
A_2  = \left(
         \begin{array}{cc}
           0 & -1 \\
           1 & 0 \\
         \end{array}
       \right),
$$
where $\sigma>1$. By the same arguments as we used before, $\lambda_\pm(\sA,\sps)=0$
for every $\sps\in\Lambda_2$. In this regard, observe that the cocycle induced by $\sA$
over the cylinder $[0;2]$ exchanges the vertical and horizontal directions,
just as in \eqref{eq.troca}. Now, it is clear that $\lambda_\pm(\sA,(1,0))=\pm\log \sigma$.
Thus, the Lyapunov exponents are discontinuous at $(\sA,(1,0))$.
\end{remark}

\subsection{Open problems}

\begin{problem}
Does continuity extend to
\begin{itemize}
\item[(a)] unbounded cocycles satisfying an integrability condition~?
\item[(b)] locally constant cocycles over Bernoulli shifts~?
\item[(c)] locally constant cocycles over Markov shifts~?
\item[(d)] locally constant cocycles in any dimension $d$~?
\item[(e)] H\"older continuous cocycles satisfying the fiber bunched condition~?
\end{itemize}
\end{problem}

In (a) we have in mind the condition $\log\|A^{\pm 1}\| \in L^1(\mu)$. Since it involves
both the cocycle and the probability measure, in this case the topology should be defined
in the space of pairs $(\sA,\sps)$. By \emph{locally constant} in (b) and (c) we mean that
$A(\bx) $ depends on a bounded number of coordinates of $\bx$. We have treated the
case when $A(\bx)$ depends only on the zeroth coordinate of $\bx$.  In (d) it suffices
to consider the largest Lyapunov exponent: then, using exterior powers in a well-known way
(see Peres~\cite{Pe91}, for instance), one would get continuity for all Lyapunov exponents.
An interesting special case to look at are symplectic cocycles, that is, such that every $A(\bx)$
preserves some given symplectic form. See \cite{AV3} for the definition of the fiber bunching
condition in (e). By Theorem~\ref{t.discontinuity} we can not expect continuity to hold for
general H\"older cocycles. On the other hand, the hypotheses of the theorem is incompatible
with fiber bunching.

\begin{problem}
Can we say more about the regularity of the Lyapunov exponents as functions of the
cocycle: H\"older continuity ? Lipschitz continuity? Differentiability?
\end{problem}

Partial answers and related results were obtained by Le Page~\cite{LP89} and Peres~\cite{Pe91}.

%\bibliography{bib}
%\bibliographystyle{plain}

\end{document}